\documentclass[11pt,openany,leqno]{article}
\usepackage[colorlinks=true]{hyperref}
\usepackage{amsmath,amsthm,amsfonts,amssymb,amscd,hyperref}
\usepackage[capitalize]{cleveref}
\usepackage{graphicx}
\usepackage{soul}

\usepackage{graphics,xcolor,url}
\numberwithin{equation}{section}
\input{xy} \xyoption{all}
\usepackage[latin1]{inputenc}
\usepackage{enumerate}
\usepackage{soul}
\usepackage{cases}
\usepackage{hyperref}
\usepackage{epsfig}
\input{xy} \xyoption{all}
\usepackage{mathrsfs}

\def \cn{\color{black} }

\DeclareMathOperator\curl{curl}

\def\qand{\quad \text{and}\quad}

\def\Diff{\mathrm{Diff}}

\def\B{\mathbb B}

\def\C{\mathbb C}
\def\R{\mathbb R}
\def\N{\mathbb N}
\def\T{\mathbb T}
\def\D{\mathbb D}

\def\I{\mathbb I}
\def\I{\mathbb I}
\def\Z{\mathbb Z}

\def \coker{\mathrm{codim}\, }

\def\leb{\mathrm{Leb}}
\def\Leb{\mathrm{Leb}}

\def\cal{\mathcal }

\def ÃÂÃÂÃÂÃÂÃÂÃÂÃÂÃÂÃÂÃÂÃÂÃÂÃÂÃÂÃÂÃÂÃÂÃÂÃÂÃÂÃÂÃÂÃÂÃÂÃÂÃÂÃÂÃÂÃÂÃÂÃÂÃÂÃÂÃÂÃÂÃÂÃÂÃÂÃÂÃÂÃÂÃÂÃÂÃÂÃÂÃÂÃÂÃÂÃÂÃÂÃÂÃÂÃÂÃÂÃÂÃÂÃÂÃÂÃÂÃÂÃÂÃÂÃÂÃÂÃÂÃÂÃÂÃÂÃÂÃÂÃÂÃÂÃÂÃÂÃÂÃÂÃÂÃÂÃÂÃÂÃÂÃÂÃÂÃÂÃÂÃÂÃÂÃÂÃÂÃÂÃÂÃÂÃÂÃÂÃÂÃÂÃÂÃÂÃÂÃÂÃÂÃÂÃÂÃÂÃÂÃÂÃÂÃÂÃÂÃÂÃÂÃÂÃÂÃÂÃÂÃÂÃÂÃÂÃÂÃÂÃÂÃÂÃÂÃÂÃÂÃÂÃÂÃÂÃÂÃÂÃÂÃÂÃÂÃÂÃÂÃÂÃÂÃÂÃÂÃÂÃÂÃÂÃÂÃÂÃÂÃÂÃÂÃÂÃÂÃÂÃÂÃÂÃÂÃÂÃÂÃÂÃÂÃÂÃÂÃÂÃÂÃÂÃÂÃÂÃÂÃÂÃÂÃÂÃÂÃÂÃÂÃÂÃÂÃÂÃÂÃÂÃÂÃÂÃÂÃÂÃÂÃÂÃÂÃÂÃÂÃÂÃÂÃÂÃÂÃÂÃÂÃÂÃÂÃÂÃÂÃÂÃÂÃÂÃÂÃÂÃÂÃÂÃÂÃÂÃÂÃÂÃÂÃÂÃÂÃÂÃÂÃÂÃÂÃÂÃÂÃÂÃÂÃÂÃÂÃÂÃÂÃÂÃÂÃÂÃÂÃÂÃÂÃÂÃÂÃÂÃÂÃÂÃÂÃÂÃÂÃÂÃÂÃÂÃÂÃÂÃÂÃÂÃÂÃÂÃÂÃÂÃÂÃÂÃÂÃÂÃÂÃÂÃÂÃÂÃÂÃÂÃÂÃÂÃÂÃÂÃÂÃÂÃÂÃÂÃÂÃÂÃÂÃÂÃÂÃÂÃÂÃÂÃÂÃÂÃÂÃÂÃÂÃÂÃÂÃÂÃÂÃÂÃÂÃÂÃÂÃÂÃÂÃÂÃÂÃÂÃÂÃÂÃÂÃÂÃÂÃÂÃÂÃÂÃÂÃÂÃÂÃÂÃÂÃÂÃÂÃÂÃÂÃÂÃÂÃÂÃÂÃÂÃÂÃÂÃÂÃÂÃÂÃÂÃÂÃÂÃÂÃÂÃÂÃÂÃÂÃÂÃÂÃÂÃÂÃÂÃÂÃÂÃÂÃÂÃÂÃÂÃÂÃÂÃÂÃÂÃÂÃÂÃÂÃÂÃÂÃÂÃÂÃÂÃÂÃÂÃÂÃÂÃÂÃÂÃÂÃÂÃÂÃÂÃÂÃÂÃÂÃÂÃÂÃÂÃÂÃÂÃÂÃÂÃÂÃÂÃÂÃÂÃÂÃÂÃÂÃÂÃÂÃÂÃÂÃÂÃÂÃÂÃÂÃÂÃÂÃÂÃÂÃÂÃÂÃÂÃÂÃÂÃÂÃÂÃÂÃÂÃÂÃÂÃÂÃÂÃÂÃÂÃÂÃÂÃÂÃÂÃÂÃÂÃÂÃÂÃÂÃÂÃÂÃÂÃÂÃÂÃÂÃÂÃÂÃÂÃÂÃÂÃÂÃÂÃÂÃÂÃÂÃÂÃÂÃÂÃÂÃÂÃÂÃÂÃÂÃÂÃÂÃÂÃÂÃÂÃÂÃÂÃÂÃÂÃÂÃÂÃÂÃÂÃÂÃÂÃÂÃÂÃÂÃÂÃÂÃÂÃÂÃÂÃÂÃÂÃÂÃÂÃÂÃÂÃÂÃÂÃÂÃÂÃÂÃÂÃÂÃÂÃÂÃÂÃÂÃÂÃÂÃÂÃÂÃÂÃÂÃÂÃÂÃÂÃÂÃÂÃÂÃÂÃÂÃÂÃÂÃÂÃÂÃÂÃÂÃÂÃÂÃÂÃÂÃÂÃÂÃÂÃÂÃÂÃÂÃÂÃÂÃÂÃÂÃÂÃÂÃÂ©{\' e}

\newtheorem{proposition}{Proposition}[section]
\newtheorem{theorem}[proposition]{Theorem}
\newtheorem*{theorem*}{Theorem}
\newtheorem{coro}[proposition]{Corollary}

\newtheorem{problem}[proposition] {Problem}
\newtheorem*{problem*}{Problem}
\newtheorem{definition}[proposition]{Definition}
\newtheorem{lemma}[proposition] {Lemma}

\newtheorem{theo}{Theorem}

\newtheorem{fact}[proposition]{Fact}

\newtheorem{conjecture}[proposition]{Conjecture}

\theoremstyle{remark}
\newtheorem{example}[proposition]{Example}
\newtheorem{remark}[proposition]{Remark}

\makeatletter

\DeclareTextFontCommand{\emph}{\em\bf}

\newcommand{\Addresses}{{
		\bigskip
		\footnotesize
		
		Pierre Berger\par\nopagebreak
		\textit{E-mail address}: \texttt{pierre.berger@imj-prg.fr}
		
		\medskip
		
		Anna Florio\par\nopagebreak
		\textit{E-mail address}: \texttt{florio@ceremade.dauphine.fr}
		
		\medskip
		
		Daniel Peralta-Salas\par\nopagebreak
		\textit{E-mail address}: \texttt{dperalta@icmat.es}
		
}}

\begin{document}
\title
{Steady Euler flows on $\R^3$ with wild and universal dynamics}
\author{Pierre Berger\thanks{IMJ-PRG, CNRS, Sorbonne University, Paris University, partially supported by the ERC project 818737 Emergence of wild differentiable dynamical systems.
}, Anna Florio\thanks{CEREMADE-Universit\'e Paris Dauphine-PSL, 75775 Paris, France. Partially supported by the project ANR CoSyDy (ANR-CE40-0014).}, Daniel Peralta-Salas\thanks{Instituto de Ciencias Matem\'aticas, Consejo Superior de Investigaciones Cient\'\i ficas, 28049 Madrid, Spain. Supported by the grants EUR2019-103821, PID2019-106715GB GB-C21, CEX2019-000904-S and RED2018-102650-T funded by MCIN/AEI/ 10.13039/501100011033}}

\date{\today}
\maketitle
\begin{abstract}
Understanding complexity in fluid mechanics is a major problem that has attracted the attention of physicists and mathematicians during the last decades. Using the concept of renormalization in dynamics, we show the existence of a locally dense set $\mathscr G$ of stationary solutions to the Euler equations in $\mathbb R^3$ such that each vector field $X\in \mathscr G$ is universal in the sense that any area preserving diffeomorphism of the disk can be approximated (with arbitrary precision) by the Poincar\'e map of $X$ at some transverse section. We remark that this universality is approximate but occurs at all scales. In particular, our results establish that a steady Euler flow may exhibit any conservative finite codimensional dynamical phenomenon; this includes the existence of horseshoes accumulated by elliptic islands, increasing union of horseshoes of Hausdorff dimension $3$ or homoclinic tangencies of arbitrarily high multiplicity. The steady solutions we construct are Beltrami fields with sharp decay at infinity. To prove these results we introduce new perturbation methods in the context of Beltrami fields that allow us to import deep techniques from bifurcation theory: the Gonchenko-Shilnikov-Turaev universality theory and the Newhouse and Duarte theorems on the geometry of wild hyperbolic sets. These perturbation methods rely on two tools from linear PDEs: global approximation and Cauchy-Kovalevskaya theorems. These results imply a strong version of V.I. Arnold's vision on the complexity of Beltrami fields in Euclidean space.
\end{abstract}

\tableofcontents

\vspace{2cm}
\section{Introduction and statements of the main results}

The evolution of an ideal fluid flow in equilibrium is described by an autonomous vector field $X$ (the velocity field of the fluid) and a scalar function $P$ (the hydrodynamic pressure) that satisfy the stationary Euler equations in Euclidean space:
\begin{equation}\label{Eq.euler}
\nabla_X X=-\nabla P\,, \qquad \text{div} X=0\,.
\end{equation}
Here $\nabla_X X$ denotes the covariant derivative of $X$ along itself and $\text{div}$ is the divergence operator, everything computed using the Euclidean metric. A vector field $X$ on $\mathbb R^3$ that satisfies Equation~\eqref{Eq.euler} for some $P$ is called a steady Euler flow. We recall that the orbits (or integral curves) of $X$ are the stream lines of the fluid, which in the stationary case coincide with the paths followed by the fluid particles.
In terms of the vorticity field $\curl X$ and the Bernoulli function $B:=P+\frac12 |X|^2$, the stationary Euler equations can be equivalently written~\cite{AK} as
\begin{equation*}
X \times \curl X=\nabla B\,, \qquad \text{div} X=0\,.
\end{equation*}

In 1965 V.I. Arnold~\cite{Ar65,Ar66} noticed that $B$ is a first integral of $X$, and hence it is invariant by its flow. Whenever $X$ is analytic, then either $B$ is constant or displays a critical set of Lebesgue measure zero. In the latter case, the celebrated Arnold's structure theorem showed that any regular level set of the Bernoulli function $B$ is a flat surface on which $X$ acts locally as a translation. Hence, the velocity is integrable in a sense analogous to Liouville integrability in Hamiltonian systems. This can be interpreted as a laminar behavior of steady Euler flows~\cite{MYZ}.

Otherwise, $\curl X=\lambda \cdot X$, for some real analytic function $\lambda$ (under the assumption that $X$ is analytic). In \cite{He66}, H\'enon observed that if $\lambda$ is not constant then the behavior of the flow is laminar. In the case that $\lambda\equiv 0$, then the dynamics $X$ on $\R^3$ must be the gradient of a harmonic function, and its entropy is zero. Hence the only (analytic) solutions of the stationary Euler equations which might have a complicated dynamics are:

\begin{definition} A vector field $X$ is $\lambda$-\emph{Beltrami} if there exists a real number $\lambda\neq 0$ such that:
\[
\curl X = \lambda\cdot  X\; .
\]
\end{definition}

Note that every Beltrami field is a solution of the steady Euler equations (with constant Bernoulli function). In the context of magnetohydrodynamics, these fields are known as force-free fields, and appear when modeling stellar atmospheres~\cite{AK}.

H\'enon's numerical experiments with the ABC flows~\cite{He66} (an explicit family of Beltrami fields on $\mathbb T^3$) showed that Beltrami fields might have a complicated dynamics (KAM tori and stochastic sea). In fact, Arnold had the following vision on the complexity of Beltrami flows:
\begin{center}\it
The flow of a vector field satisfying $\curl X= \lambda \cdot  X$ probably displays stream lines with topologies as complicated as those of orbits in celestial mechanics (see \cite[Fig. 6]{Ar63}).
\end{center}
\begin{flushright}
Arnold  \cite[Rk. P. 90]{Ar65} and  \cite[P. 347]{Ar66}.
 \end{flushright}
 \begin{figure}[h]
	\centering
	\includegraphics[scale=0.25]{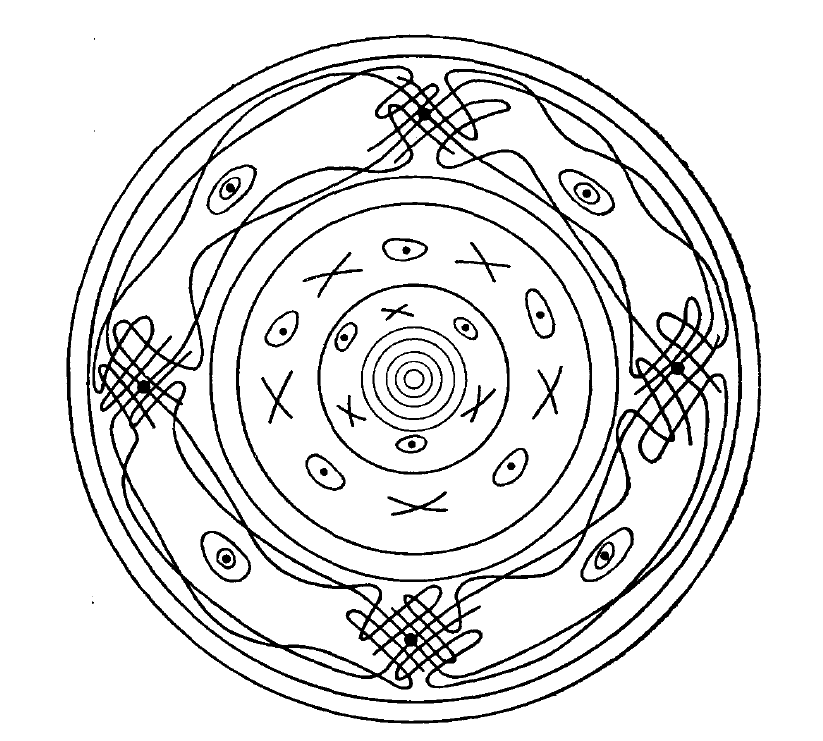}
	\caption{\cite[Fig. 6]{Ar63}}
	\label{ImageArnold}
\end{figure}
As commented around \cite[Fig. 6]{Ar63}, by  ``complicated  topology", Arnold meant the coexistence of nested (KAM) tori centered at elliptic orbits, such that in between there are elliptic and saddle orbits, the elliptic ones being surrounded by (KAM) tori and the stable and unstable manifolds of the saddle ones intersecting each other to create an intricate network in the ``zones of instability".
It is relevant to observe that the study of the complexity of Beltrami fields may be of interest to study complex non-stationary solutions of the Navier-Stokes equations. Indeed experimental and numerical observations postulated
a phenomenon known as ``Beltramization'', which states that  in turbulent regions the velocity $X$ tends to align to the vorticity $\curl X$ when the Reynolds number is large~\cite{Farge,Monchaux}. 
\medskip

As we will discuss in \textsection\ref{state of the art}, the behavior presented in Figure~\ref{ImageArnold} is well known to be typical among smooth conservative surface dynamics; in contrast, the previous literature on Beltrami fields yields dynamical phenomena far from establishing Arnold's picture. 
Our main results imply a strong version of Arnold's description. We will show, in particular, the existence of a non-empty open subset $\cal N_{\mathscr B}(\R^3)$ of Beltrami fields in which any (families of) conservative dynamics from the disk is arbitrarily well approximated by (families of) Poincar\'e return maps of generic (families of) Beltrami fields in $\cal N_{\mathscr B}(\R^3)$. The main difficulty of our proof is to design perturbations of the homoclinic tangle to create multiple homoclinic tangencies, combining the very rigid space of Beltrami fields with the wild geometry of the Newhouse domain.

\medskip

Let us precise the statements of our main results. Denote by $\D$ the closed unit disk of $\R^2$. Let $\Diff^\infty_{\Leb+}(\D, \R^2)$ be the set of diffeomorphisms from $\D$ onto their image in $\R^2$ which are infinitely smooth and symplectic. By \emph{symplectic}, we mean orientation and area preserving.

Let $U\subset \R^3$ be an open subset of $\R^3$.  A vector field on $U$ is of class $C^\infty_\leb$ if it is smooth and volume preserving.
Let us denote by $\Gamma_\leb^\infty(U)$ the space of $C^\infty_\leb$-vector fields on $(U,\leb)$.
Let $\mathscr B(U)\subset \Gamma_\leb^\infty(U)$ be the space of (Euclidean) $1$-Beltrami fields on $U$ with sharp decay. These are vector fields $X$ such that $\mathrm{curl}\, X=X$ and for any multi-index $\alpha \in \N^3$, the norm  $\|D^\alpha X(x)\|$ is dominated by $(1+\|x\|)^{-1}$ for every $x\in U$. It is clear that there is no loss of generality in choosing the constant $\lambda=1$ because the proportionality factor of a Beltrami field can be rescaled by a homothety. In what follows,  \emph{Beltrami field}  will refer canonically to a $1$-Beltrami field.

We recall that a saddle periodic orbit $P$ of a vector field $X_0$ displays a \emph{homoclinic tangency} if its stable and unstable manifolds are tangent at a point $Q$. The tangency is \emph{quadratic} if the curvatures of the stable and unstable manifolds at $Q$ are different. An \emph{unfolding} of $X_0$ is a    smooth $k$-parameter family $(X_p)_{p\in\mathbb{I}^k}$ containing $X_0$ at $p=0$, where $\mathbb{I}:=[-1,1]$. If it is a one-parameter family ($p$ varies in an interval $\mathbb{I}$), the unfolding is \emph{non-degenerate} if the relative position of the continuation of the  stable and unstable manifolds at the tangency point $Q$ has non-zero derivative at $p=0$.  Then we say that the quadratic homoclinic tangency  
 \emph{unfolds non-degenerately}.

Let $\cal N_\Gamma(\R^3)$ be the interior of the closure\footnote{\label{footnote1}See  \textsection \ref{def of topology} for the definition of the involved topology.}
of the $C^\infty_\leb$-vector fields on $\R^3$ exhibiting a periodic saddle displaying a homoclinic tangency. A consequence of Newhouse's celebrated theorem \cite{Ne1} is that $\cal N_\Gamma(\R^3)$ is non-empty.
A consequence of Duarte's Theorem \cite{Du99}
is that the closure of $\cal N_\Gamma(\R^3)$ is equal to the closure of the set of $C^\infty_\leb$-vector fields on $\R^3$ for which there exists a periodic saddle displaying a homoclinic tangency. The open set $\cal N_\Gamma(\R^3)$ is the so-called \emph{Newhouse domain}. It is formed by dynamics which are extremely rich as we will see in the sequel.

Our first main result is the Beltrami counterpart of these theorems.
Let $\cal N_{\mathscr B}(\R^3)$ be the interior of the closure$^{\ref{footnote1}}$ 
of the set of Beltrami fields $X\in \mathscr B(\R^3)$ that exhibit a saddle periodic orbit displaying a homoclinic tangency \emph{and} such that the tangency can be non-degenerately unfolded by a family in $\mathscr B(\R^3)$. Here is our first main result:

\begin{theo}\label{ThmA}
The set $\cal N_{\mathscr B}(\R^3)$ is non-empty.
\end{theo}

The main difficulty of this result is to show the existence of a Beltrami field $X\in \mathscr B(\R^3)$ displaying a  homoclinic tangency  which can be   non-degenerately unfolded  in $\mathscr B(\R^3)$. Then  Duarte's theorem implies
 directly the non-emptiness of
$\cal N_{\mathscr B}(\R^3)$.
In particular \cref{ThmA} provides the \emph{first} example of a Beltrami field in Euclidean space with a \emph{quadratic} homoclinic tangency.  Examples of homoclinic tangencies exist on the torus $\R^3/\Z^3$, see e.g.~\cite[ABC flow Example 1.9, Page 76]{AK}; yet the manifold is not $\R^3$ and the example is integrable so the tangency is not quadratic. Moreover the displayed non-degenerate quadratic unfolding implies the existence of Beltrami fields with horseshoes (without any computer assisted estimate as in \cite{EPSR}). Also by Duarte's theorem, this reveals, as a new phenomenon in the Beltrami setting,  the coexistence of infinitely many KAM tori which accumulate on a Smale horseshoe. These consequences will be discussed in detail in \textsection\ref{state of the art}.

Our definition of the Beltrami Newhouse domain $\cal N_{\mathscr B}(\R^3)$ asks for the non-degeneracy of the unfolding of quadratic homoclinic tangencies, which is not asked in the standard definition of the smooth Newhouse domain  $\cal N_{\Gamma}(\R^3)$. Indeed it is well known that a homoclinic tangency in $\Gamma^\infty_\leb(\R^3)$ can be non-degenerately unfolded \cite{BT86}. A natural  Beltrami counterpart of
this result is:
\begin{conjecture}
Every homoclinic tangency for a periodic saddle orbit of a Beltrami field $X\in \mathscr B(\R^3)$ unfolds non-degenerately in $\mathscr B(\R^3)$.
\end{conjecture}
We believe that this conjecture is the main issue to prove that conservative Kupka-Smale systems are dense among Beltrami fields. Such density has been shown for other rigid classes of dynamics, see \cite{BT86,BHI05}.

\medskip

Now let us shed light on the richness of the generic dynamics in $\cal N_\mathscr B(\R^3)$. In short, our second main  \cref{ThmB} states that a dense subset of $\cal N_\mathscr B(\R^3)$ is universal: its Poincar\'e return maps realize a dense subset of conservative dynamics from the disk. This will imply many new phenomena, including Arnold's vision\footnote{Theorem B is  stronger than Arnold's vision: there is no example of universal dynamics in celestial mechanics.}.
\begin{definition}\label{def renormalized ietration}
A \emph{renormalized iteration} of a $C_\leb^\infty$-vector field $X$ on an open subset  $U\subset \R^3$ is
a Poincar\'e return map $G$ at a transverse disk in rescaled coordinates. More precisely, there is an embedding $\phi : \R^2 \hookrightarrow U$,
$\Sigma :=\phi (\R^2)$ is transverse to $X$, such that the Poincar\'e  map $  F$ from $\check \Sigma :=\phi (\D)$ into $\Sigma$ is well defined and
the map $  \phi ^{-1}\circ F\circ \phi|_\D $ is symplectic and equals $G$.
 \end{definition}

\begin{definition}[Compare to \cite{BD00,Tu03,GST07}]\label{def universal}
	A $C_\leb^\infty$-vector field $X$ on $U\subset \R^3$ is \emph{universal} if  the set of
renormalized iterations of $X$ is dense in $ \Diff^\infty_{\Leb+}(\D, \R^2)$.
The vector field $X$ is \emph{compactly universal} if there exists an open bounded subset $U'\subset U$ such that $X|_{U'}$ is universal.
\end{definition}

Gonchenko-Shilnikov-Turaev's Theorem \cite{GST07} implies that a dense subset of $\cal N_\Gamma(\R^3)$
is formed by compactly universal dynamics. We propose the counterpart of this result in the much more rigid world of Euclidean Beltrami fields:
\begin{theo}\label{ThmB}  There is a topologically generic\footnote{A topologically generic subset is a set equal to a
countable intersection of open and dense subsets.} subset $\mathscr G$ of $\cal N_\mathscr{B}(\R^3)$ formed by compactly universal vector fields.  \end{theo}
Let us introduce a generalization of the above definition and result for families of dynamics.  
For $J\ge 0$, let
$ \Diff^\infty_{\Leb+}(\D, \R^2)_{\I^J}$ denote the space of smooth families $(f_p)_{p\in \I^J}$ of maps $f_p\in  \Diff^\infty_{\Leb+}(\D, \R^2)$. We endow this space with the topology given by the canonical  inclusion:
\[  \Diff^\infty_{\Leb+}(\D, \R^2)_{\I^J}\hookrightarrow   C^\infty(\I^J\times \D, \R^2).\]

\begin{definition}\label{def para universal}
Let $\cal F$ be a Fr\' echet space formed by smooth conservative flows of $\R^3$.

 A subset of $\mathscr E\subset\cal F$ is \emph{$J$-para-universal} for $J\ge 0$  if, for any non-empty open subsets $\mathcal V$ of $\mathscr E$ and    $\cal O_J$ of  $\Diff^\infty_{\Leb+}(\D, \R^2)_{\I^J}$,   there exists a smooth family $(X_p)_{p\in \I^J}$  formed by vector fields in $\mathcal V$ and  such that   a smooth family of renormalized iterations of $(X_p)_{p\in\I^j}$ is in $\cal O_J$.

 The subset $\mathscr E\subset \cal F$ is \emph{para-universal} if  it is   $J$-para-universal for every $J\ge 0$.

  The subset $\mathscr E\subset \cal F$ is \emph{compactly para-universal}  if there exists a bounded open subset $U\subset \R^3$ such that
  $\mathscr E$ satisfies the condition of  para-universality with  orbits of the renomalizations included in $U$.

%
%
%
%
%
\end{definition}
\begin{remark}\label{para universal are universal}
If $\mathscr E$ is a Baire space and is (resp. compactly) para-universal, then a topologically generic subset $\mathscr G\subset \mathscr E$ is formed by (resp. compactly) universal vector fields. Indeed, by $0$-para-universality of every non-empty open subset of $\mathscr E$,  for every non-empty open set $\cal O$ of  $\Diff^\infty_{\Leb+}(\D, \R^2)$, there is a dense subset $\mathscr D$ of vector fields in $\mathscr E$ which display a Poincar\'e return map in $\cal O$. As $\cal O$ is open, $\mathscr D$ can be taken open and dense. We conclude by using that $\Diff^\infty_{\Leb+}(\D, \R^2)$ has a countable base of open subsets.   \end{remark}

By the latter remark, the following  generalizes \cref{ThmB} and enables to capture any finite codimensional phenomenon of $\Diff^\infty_{\Leb+}(\D, \R^2)$ in a small open set of Beltrami fields contained in $\mathcal{N}_\mathscr{B}(\R^3)$. In particular, this implies that any bifurcation of conservative maps of the disk that is a finite codimensional phenomenon appears in any open subset of $\cal N_{\mathscr B}(\R^3)$ as a renormalization.

\begin{theo}\label{ThmC}
The subset  $\cal N_{\mathscr B}(\R^3)$ is compactly para-universal.
\end{theo}

For the proof of these theorems, we introduce new general perturbation methods in hydrodynamics allowing to join   PDE techniques and  deep results from bifurcation theory in dynamical systems.
These perturbation methods are implemented for 3-dimensional Beltrami flows and start with two tools from PDEs: global approximation and Cauchy-Kovalevskaya theorems. These tools are rather general and appear not only in Beltrami fields but also in any linear second order elliptic or parabolic equation with analytic coefficients \cite{Bro62,E-GF-PS}.

Let us also mention that the origin of the global approximation theorems in PDEs is Runge's theorem in complex analysis. The latter was used by Buzzard \cite{Bu97} to find the first example of a wild (conservative) polynomial automorphism of $\C^2$. Hence it is natural to think that the methods presented in this work could be (non-trivially) adapted to prove a complex and discrete time counterpart of our main theorem: the local density of compactly universal maps among transcendental conservative complex automorphisms of $\C^2$. In this setting, universal means that any complex and conservative polynomial automorphism can be approximated by a  renormalized iteration of the dynamics.

\medskip

To prove Theorems \ref{ThmB} and \ref{ThmC}, we will extract from the Gonchenko-Shilnikov-Turaev theory \cite{GST07} that it suffices to establish the existence of  a non-degenerate unfolding of a multiple homoclinic tangency (see \cref{multi homo}). To show such an existence, we will make use of the geometry of wild hyperbolic sets (see \cref{def wild}) unveiled by the seminal works of Newhouse and Duarte \cite{Ne2,Du} to implement a new perturbative method based on the global approximation theorem \cite{EncPS,EncPS15} (which is used as a black box). Indeed, the main difficulty will be to duplicate a homoclinic unfolding in the very rigid class of Euclidean Beltrami fields. In order to do so we will have to display disjoint and homoclinically related horseshoes which can be separated by non nested compact surfaces. See \cref{creation disjoint wild horseshoe}.

To prove \cref{ThmA}, we will give a self contained extension of the Cauchy-Kovalevskaya theorem for curl of \cite{EncPS} in \cref{section CK and GAT}, which we will use (again as a black box) to unfold heteroclinic links (see \cref{sketch thm A}). Then the main difficulty will be to handle this simultaneously for two heteroclinic links, while interferences appear and  might prevent from the creation of the sought unfolding. To perform this we will use a new linear algebra trick applied to an inverse of the Melnikov operator. See \cref{sec:LinALg}.

\section{Contributions to the state of the art on complexity of steady Euler flows}
\label{state of the art}
Showing complexity in fluid mechanics is a long standing problem.
It has inspired many mathematicians and physicists such as Ruelle, Takens \cite{RT71}, Newhouse \cite{NRT78},  Feigenbaum \cite{Fe79}, Coullet, Tresser and Arneodo \cite{CTA80}, who have proposed a link between complexity in fluid mechanics and renormalization in dynamical systems. One aspect of this program is to understand the renormalized iterations of a fluid flow. For instance, \cite{RT71} stated that the Navier-Stokes equations in the space of velocity fields could have a renormalization displaying a strange attractor. This also leads to the famous problem on the typicality of the set of dynamical systems which can be
obtained after renormalization of a close to the identity map. This was also influential for the works  \cite{BD00,Tu03,GST07}, which set up the concept of universality in dynamics.

  In more flexible spaces than the Euclidean 3-dimensional one, let us mention the work \cite{Torres} of Torres de Lizaur who constructed an embedding of any smooth flow into the time-dependent evolution of the Euler PDE on high-dimensional Riemannian manifolds. See also \cite{Vak21} for other embedding results in the context of the Oberbeck-Boussinesq model of the Navier-Stokes equations (on $2$-dimensional domains) with external parameters (a space-dependent external force and heat source).

As we mentioned in the previous section, we use the most recent mathematical developments in bifurcation theory to establish the universality, and so the extreme complexity, of some steady fluid flows. More generally, the present work is the first which shows the existence of an Euclidean (not forced) fluid motion which is universal (in the sense of Definition~\ref{def universal}). As we saw, this question goes back, at least, to the foundational works of Arnold \cite{Ar65,Ar66}, and was numerically explored by several authors \cite{He66}. Another notion of universality based on Turing completeness  has been promoted by the recent Tao's program \cite{Tao18} to address the blow-up problem for the Euler and Navier-Stokes equations: this is a priori independent from the universality studied in the present paper.

Up to now, the previous works on dynamical phenomena exhibited by Beltrami fields in Euclidean space $\R^3$ include:
\begin{enumerate}
\item Periodic orbits (elliptic or hyperbolic) and invariant tori of arbitrary topology \cite{EncPS,EncPS15}.
\item Invariant tori bounding regions that contain any prescribed number of hyperbolic orbits (but possibly not homoclinically related) \cite{Advances}.
\item Heteroclinic connections and Smale horseshoes, as an auxiliary result of \cite{EPSR}.
\item A universal Turing machine contained in an invariant plane, see~\cite{CMPS21}.
\end{enumerate}
It is also known that these structures occur with probability $1$ as a consequence of the theory of Gaussian random Beltrami fields \cite{EPSR}. While a union of these objects forms Arnold's picture, we emphasize that, in general, the previous techniques do not allow to embed them in a Beltrami field in such a way that the given relative position are respected (for instance, having the horseshoe inside the torus).

In particular, it was missing to construct an Euclidean Beltrami field exhibiting Arnold's picture shown in \cref{ImageArnold} as a Poincar\'e section. In contrast, this vision is now a trivial consequence of \cref{ThmB}.
Indeed by KAM theorem, Arnold's picture is persistent by $C^\infty_\Leb$-perturbations while
 \cref{ThmB}  shows that a locally generic Beltrami field  has its set of renormalizations which is dense in the space of conservative  maps from the disk to the plane.

Moreover \cref{ThmC} goes beyond Arnold's vision in the sense that it establishes the existence of parametric families of Beltrami fields whose corresponding families of renormalized iterations are dense in the set of families of conservative diffeomorphisms from the disk into the plane. For instance, as a consequence of
Theorems \ref{ThmA} and
\ref{ThmC} some of the dynamical phenomena that can arise in Beltrami fields are:
\begin{enumerate}
\item By Duarte's theorem \cite{Du}, there exists a topologically generic subset $\cal G$ of $\cal N_{\mathscr{B}}(\R^3)$ formed by vector fields exhibiting a horseshoe accumulated by non-degenerate elliptic periodic orbits.
\item By Gorodetski's theorem \cite{Go12} the subset $\cal G$ can be chosen so that for every vector field in $\cal G$ there exists an increasing sequence of such horseshoes whose Hausdorff dimension converges to $3$.
\item  There exists a dense subset $\cal D$ in $\cal N_{\mathscr{B}}(\R^3)$ formed by vector fields exhibiting homoclinic tangencies of arbitrarily high multiplicity.
\item For any conservative Taylor expansion $J$ there exists a dense subset in $\cal N_{\mathscr{B}}(\R^3)$ formed by vector fields with a renormalized iteration exhibiting a fixed point with Taylor expansion equal to $J$.
\item 
There is a dense subset of  $\cal N_{\mathscr{B}}(\R^3)$ formed by Beltrami field $X$ with the following property.  There exists a saddle orbit $O$  such that for every $k\ge 0$, there exists a non-degenerate elliptic orbit $O_k$  so that $(O_k)_{k\ge 0}$ converges to a compact set containing $O$ and the period of $O_k$ is asymptotic to a certain constant plus  $k$-times the period of $O$  \cite{GS01rus,GS03,GG09}. Moreover, there is a  $2$-codimensional (Fr\' echet) manifold in $\cal N_{\mathscr{B}}(\R^3)$ on which all  the elliptic orbits $O_k$  and the saddle orbit $O$ persist and satisfy the same asymptotic.
\end{enumerate}

\cref{ThmC} also implies that any conservative finite codimensional dynamical phenomena appear densely in the Newhouse domain $\cal N_{\mathscr{B}}(\R^3)$ of Beltrami fields. Let us emphasize that \cref{ThmC} provides more a lower bound on the richness and complexity of the Beltrami dynamics, rather than a complete understanding of it: a ``typical'' Beltrami field in $\cal N_{\mathscr{B}}(\R^3)$ encompasses the dynamical complexity of a dense subset of $\Diff^\infty_{\leb+}(\D, \R^2)$. Additionally, as universality was used in \cite{BerTur} to prove Herman's positive entropy conjecture (see also \cite{Ber21}), our techniques might be useful to address the following open problem:
\begin{problem}
Show that there exists an Euclidean Beltrami field with positive metric entropy, in the sense that there exists a bounded invariant set of positive Lebesgue measure formed by points with positive Lyapunov exponent.
\end{problem}
Another open problem is that of the statistical complexity of a typical Beltrami field, which is measured by the emergence \cite{Ber17}. This complexity has been shown to be huge for a generic subset of $\Diff^\infty_{\leb+}(\D, \R^2)$ \cite{BB21}:
\begin{problem}
Show that there exists an Euclidean Beltrami field with a high emergence, in the sense that there exists a bounded invariant set of positive Lebesgue measure so that the ergodic decomposition of the flow restricted to this set is infinite dimensional.
\end{problem}
More generally, a natural question that remains open is whether there exists a dynamics of
$\Diff^\omega_{\leb+}(\D, \R^2)$ which is not a renormalized iteration of an Euclidean Beltrami field. On the other hand, it is well known that not every analytic volume preserving flow is orbit equivalent to a Beltrami field, see e.g. \cite{CV17}.

In strong contrast, using the flexibility of contact geometry, it was shown in \cite{PNAS} that, if one allows to deform the metric on $\R^3$ (no longer Euclidean), then any dynamics of $\Diff^\infty_{\leb+}(\D, \R^2)$ is a renormalized iteration of a Beltrami field for some adapted Riemannian metric. In the context of high-dimensional fluids, a new $h$-principle for Reeb flows allows one to embed any finite dimensional dynamics (possibly dissipative) as an invariant subset of a Beltrami field for some special Riemannian metric \cite{Cardona,CMPP}.

We finish this section recalling that every finite dimensional family of Euclidean Beltrami fields can be approximated by (the localization) of a finite dimensional family of Beltrami fields on the flat torus $\T^3$ (or respectively on the round sphere $\mathbb S^3$) \cite{EPSTL17}. Accordingly, a direct implication of \cref{ThmC} is that Beltrami fields on $\T^3$ or $\mathbb S^3$ may exhibit (in particular) any finite codimensional phenomena of $\Diff^\infty_{\leb+}(\D, \R^2)$.

\section{Structure of the proof}
\subsection{Perturbations techniques in the space of Beltrami fields}
So far the study of bifurcations of Beltrami fields were performed using purely analytic methods, at the unfolding of some integrable systems; for example, in \cite{Advances} the authors used the subharmonic Melnikov technique to bifurcate hyperbolic periodic orbits from resonant invariant tori. In this work we introduce new geometric methods to perturb systems which may be non-integrable.

A main issue to develop a general perturbation theory of (Euclidean) Beltrami fields is that we cannot localize the perturbations as one does in differentiable dynamical systems. Although this problem appears also for real analytic dynamics, these form a dense subset of differentiable dynamics, and likewise for the corresponding spaces of parameter families. So one can easily transfer any finite codimensional phenomenon from the differentiable setting to the real analytic setting. In contrast, Beltrami fields are not dense in the space of smooth conservative vector fields (the PDE equation is an infinite codimensional condition). The following tool will be used as an ersatz for local perturbation techniques:

\begin{theorem}[Global Approximation Theorem {\cite[Thm 8.3]{EncPS15}}]\label{GAT}
Let $K\subset \R^3$ be a compact set whose complement is connected. Let $X$ be a Beltrami field on an open neighborhood of $K$. Then for every $N\ge 1$, there exists a    Beltrami field  $ \tilde X $ with $\tilde X \in \mathscr B(\R^3)$ such that:
\[\| D^\alpha X(z)- D^\alpha \tilde X(z)\|<\frac1N \; ,\quad \forall \alpha \in \{0,1,\dots, N\}^3\qand \forall z\in K\; .\]
\end{theorem}
Let us emphasize that the Global Approximation Theorem does not provide any control on the extended vector field $\tilde X$ on the complement of a neighborhood of $K$.

To increase the multiplicity of the unfolding of homoclinic tangencies, we will use the following immediate (although fruitful)  consequence of Theorem \ref{GAT}:

\begin{lemma}\label{GATcoro}
Let $K_+$ and $K_-$ be two disjoint compact subsets of $\R^3
$ with connected complement. Let $X$ be a Beltrami field on an open neighborhood of $K=K_+\sqcup K_-$.
 Then for every $N\ge 1$, there exist two  Beltrami fields  $ X_+ $ and $  X_-$
 in $ \mathscr B(\R^3)$ such that for any $\pm\in \{-,+\}$ and any  $\alpha\in \{0, \dots, N\}^3$:

 \[\| D^\alpha X_\pm  (z)- D^\alpha   X (z)\|<\frac1N \; ,\quad  z\in K_\pm \qand \| D^\alpha  X_\pm (z)\|<\frac1N \; \quad   z\in K\setminus K_\pm\; .\]
\end{lemma}
\begin{proof}
Let $V$ be the domain where the local Beltrami field $X$ is defined. It is an open neighborhood of $K$ and so it can be split into a disjoint union $V=: V_+\sqcup V_-$, where $V_+$ and $V_-$ are neighborhoods of $K_+$ and $K_-$ respectively.
For $\pm\in \{-,+\}$, to obtain the desired Beltrami field $X_\pm$, it suffices to apply \cref{GAT} to the local Beltrami field $1\!\! 1_{V_\pm}\cdot X$ on $V$.  To apply \cref{GAT} we only need to observe that $K$ has connected complement because it is the union of two disjoint compact subsets of $\R^3$ with connected complement. As we could not find any reference for this fact, we prove it in \cref{appendix classical topo}. \end{proof}
Interestingly, this corollary will lead us to explore the geometry of wild sets in order to isolate disjoint invariant compact subsets at which the unfolding will be split.
\medskip

While the Corollary \ref{GATcoro} creates two perturbations from a local Beltrami field, the next theorem provides a method to create Beltrami fields with a given normal component on a cylindrical surface:

\begin{theorem}\label{main thm} Let $\Gamma$ be an  analytic surface in $\R^3$  whose closure $\bar \Gamma$ is diffeomorphic to $\R/\Z\times \I$. For any nonempty open set $\mathcal N$ of  $C^\infty(\bar \Gamma,\R)$,  there exists a vector field $\tilde X\in \mathscr B(\R^3) $ such that the normal component of $\tilde X|_\Gamma$ on $\Gamma$ is in $\mathcal N$.
\end{theorem}
\cref{main thm} is proved below using  \cref{GAT} and  the following:
\begin{theorem} [Cauchy-Kovalevskaya's Theorem for curl] \label{cor:CK}
Let $\Gamma$ be an analytic surface in $\R^3$  whose closure $\bar \Gamma$ is diffeomorphic to $\R/\Z\times \I$.
For every $g\in C^\omega(\bar \Gamma,\R)$ there exists a neighborhood $\Omega$ of $\bar \Gamma$ in $\R^3$ and a Beltrami field $X\in \mathscr B(\Omega)$ such that the normal component of $X|_\Gamma$ on $\Gamma$ is equal to $g$.
\end{theorem}
\cref{cor:CK} will be proved in \cref{section CK and GAT} using the  Cauchy-Kovalevskaya Theorem as in \cite[Theorem 3.1]{EncPS}.
\begin{proof}[Proof of \cref{main thm}] By density of the space of  analytic maps, there exists $g\in  C^\omega(\bar \Gamma,\R)\cap \cal N$.  Then we apply Theorem \ref{cor:CK} with $g$, which provides a local Beltrami field $X$, whose normal component on $\Gamma$ is equal to $g$, which in turn can be approximated by a Beltrami field $\tilde X\in \mathscr B(\R^3)$ on account of
\cref{GAT} (obviously, the complement of a neighborhood of $\bar \Gamma$ is connected). Observe that the normal component of $X|_{\bar \Gamma}$ is close to $g$ and so it belongs to  $\cal N$.
 \end{proof}
\cref{GAT,main thm} are the only results from PDE techniques. These perturbation tools will be used as black boxes. The remaining of the proof will be purely dynamical.

\subsection{Structure of the proof of \cref{ThmA}: existence of a quadratic, non-degenerate homoclinic  tangency unfolding in the class of Beltrami fields}
\label{sketch thm A}
We will show that there are quadratic homoclinic tangencies unfolding non-degenerately
nearby a Beltrami field that displays a double strong heteroclinic link. Let us define this configuration. Let $U\subset \R^3$ be a bounded open subset of $\R^3$. A vector field $X\in \Gamma^\infty_{\mathrm{Leb}}(U)$ displays a \emph{heteroclinic link} if there are two saddle periodic orbits $\gamma^+$ and $\gamma^-$, contained in $U$, such that an unstable branch 
 of $\gamma^-$ is equal to a stable branch of $\gamma^+$ and it is still contained in $U$.  Denote by $\Gamma$ the intersection of these branches.
  The heteroclinic link $\Gamma$  is \emph{strong}   if   $\Gamma$ is orientable\footnote{In particular, the closure $\bar \Gamma$ is diffeomorphic to $\R/\Z\times [-1,1]$.} and the strong stable foliation of $\gamma^+$ in $\Gamma$    coincides with the strong unstable foliation of $\gamma^-$ in    $ \Gamma$.  See \cref{ImageSSU}. Observe that the dynamics on a strong heteroclinic link is conjugate to a product of a rotation with a gradient on the interval.

   \begin{figure}[h]
  	\centering
  	\includegraphics[scale=0.4]{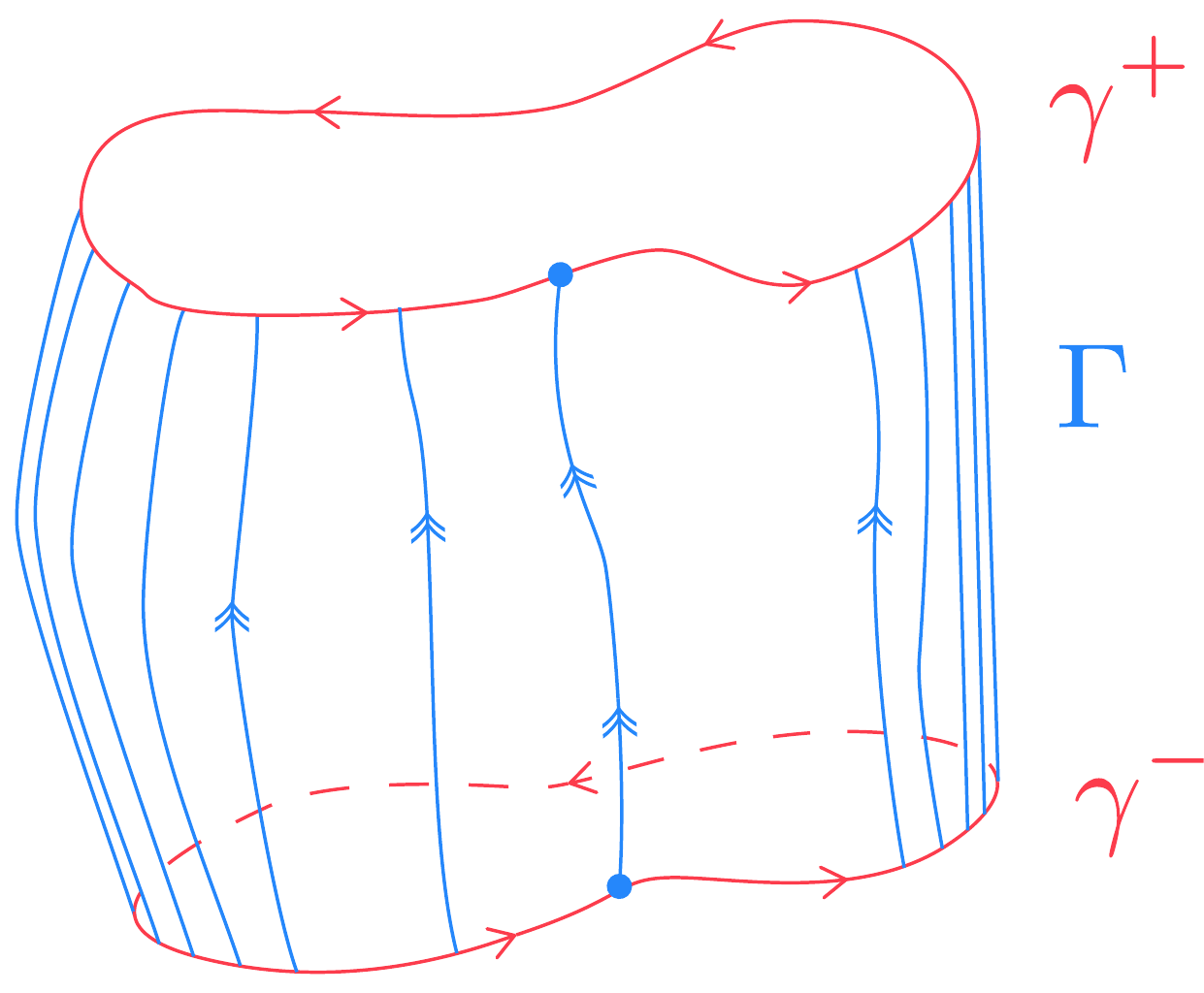}
  	\caption{A strong heteroclinic link.}
  	\label{ImageSSU}
  \end{figure}
%

  A \emph{double, strong heteroclinic link} is a pair $(\Gamma^+, \Gamma^-)$ of strong heteroclinic links between the same saddle periodic orbits $(\gamma^+,  \gamma^-)$ and so that $\Gamma^+\subset U$ and $\Gamma^-\subset U$ contain a stable branch of respectively $\gamma^+$ and $\gamma^-$. See \cref{Fig Sigma}.

Here is the only example available in the literature of a Beltrami field on an open subset of $\R^3$ displaying a \emph{double strong heteroclinic link}.
\begin{example}[
{\cite[\textsection 5 formula (5.1)]{EPSR}}]
\label{example:EPSR} The following vector field $X$ (written in cylindrical coordinates)
is  in $ \mathscr{B}(\R^3\setminus \{(0,0)\} \times \R)$  and displays a double,  strong  heteroclinic link:
	\begin{equation}
	X:= \dfrac{1}{r}\left( \partial_r\psi\, E_z -\partial_z\psi\, E_r+\dfrac{\psi}{r}\, E_\theta \right)\, ,
	\end{equation}
	where $\psi(z,r):=\cos(z)+3rJ_1(r)$ with $J_1$ being the Bessel function of the first kind and order $1$. The bounded open set $U$, with connected complement, can be taken to be any proper subset contained in the intersection of $\R^3\setminus \{(0,0)\} \times \R$ with a large enough ball.
	
		\begin{figure}[h]
		\centering
		\includegraphics[scale=0.5]{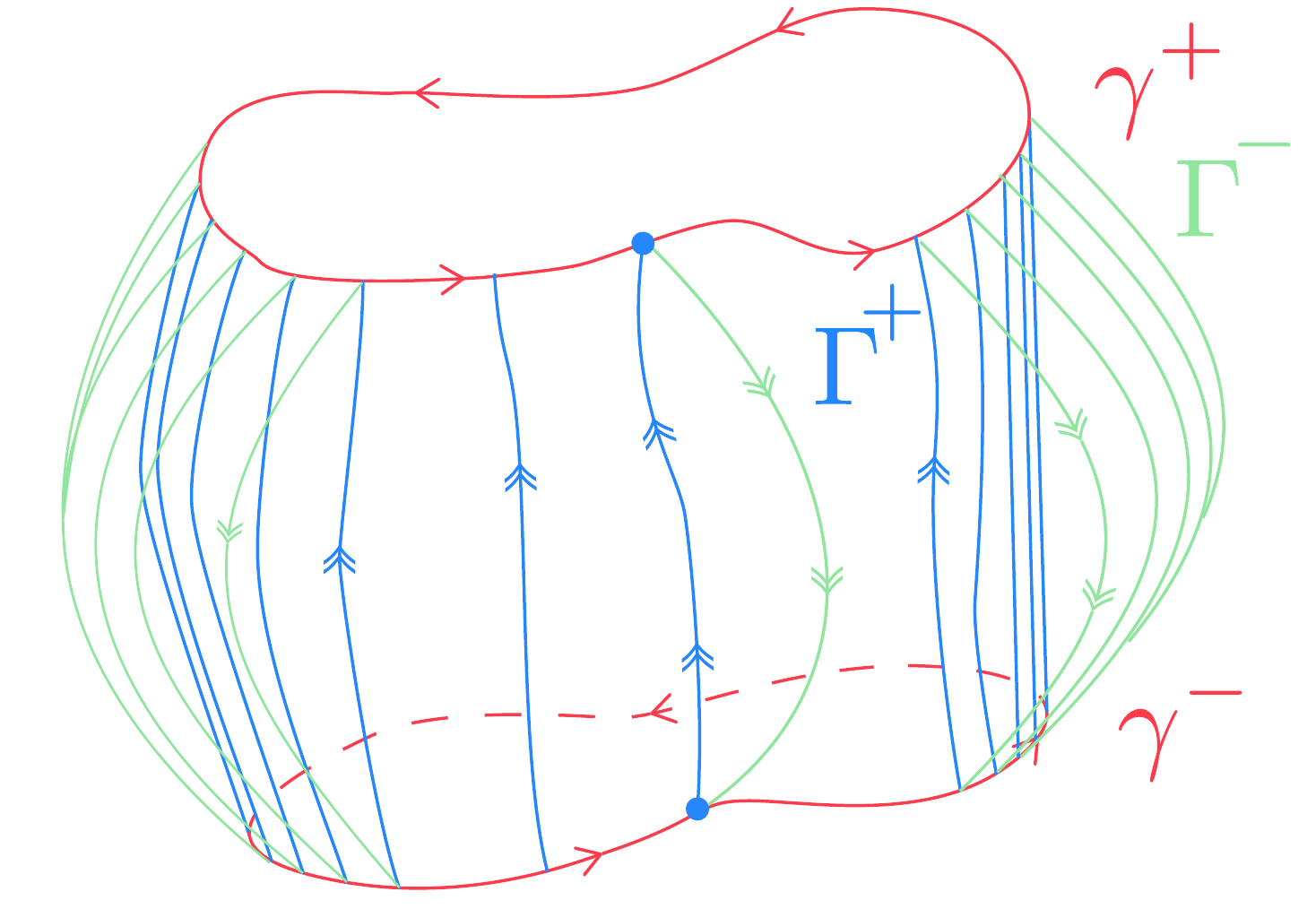}
		\caption{Periodic orbits $\gamma^\pm$ displaying a double (strong) heteroclinic link $(\Gamma^+,\Gamma^-)$.}
		\label{Fig Sigma}
	\end{figure}
\end{example}

In order to prove  \cref{ThmA} we will show that any  Beltrami field $X\in\mathscr{B}(U)$, where $U\subset \R^3$ is a bounded open subset of $\R^3$ such that $\R^3\setminus U$ is connected, displaying a double strong heteroclinic link gives rise to a smooth family $(X_p)_{p\in \I}$ of Beltrami fields in $\mathscr{B}(\R^3)$ which unfolds non-degenerately a quadratic homoclinic tangency. This will be a particular case of \cref{thm hom tg} where non-degenerate unfoldings of higher multiplicities (see \cref{hom tg higher mult}) are also constructed.

%
%
%

We present now the main tools and the key points for the proof of \cref{thm hom tg}, which states:

\begin{theorem}\label{thm hom tg}
\textit{For any  $k\geq 0$, there exists a family $(X_p)_{p\in\mathbb{I}^k}$ of Beltrami fields with $X_p\in\mathscr{B}(\R^3)$ that unfolds non-degenerately a homoclinic tangency of multiplicity $k$ at $p=0$.}
\end{theorem}

Its proof is divided in four steps.

\subsubsection{Step 1: The inverse of the Melnikov operator}
Let $U\subset \R^3$ be a bounded open subset of $\R^3$ and let $X\in\Gamma^\infty_\leb(U)$ display a strong heteroclinic link $\Gamma\subset U$ between periodic saddle orbits $\gamma^+$ and $\gamma^-$.
We recall that the closure $\bar L$ of the half strong stable manifold of a point in  $\gamma^+$ in $\Gamma$ 
 is diffeomorphic to a segment. Thus there is a disk  $\Sigma\subset U$ 
  which contains $\bar L$  and   intersects transversally   $\Gamma$.
Up to reducing $\Sigma$, assume that $\Sigma$ is transverse to the vector field $X$.

 Observe that $\Sigma\cap\Gamma$ is a heteroclinic link for the Poincar\'e map $f$ at $\Sigma$. A small perturbation $W\in \Gamma^\infty_\leb(\R^3)$ of the vector field $X$ induces a small perturbation  $f_W$ of the Poincar\'e return map on $\Sigma$. In order to study the unfolding of the heteroclinic link $\Gamma$, we shall analyze the distance between the stable and unstable manifolds of the perturbed Poincar\'e return map at $\Sigma$ using the \emph{Melnikov operator}.

Roughly speaking (see  \cref{section Melnikov} for rigorous details), there is a chart of a neighborhood of $L$ in $\Sigma$ which sends $L$ to $\R\times \{0\}$ and so that seen in this chart the dynamics is the translation by $(1,0)$. Then for a perturbation of $f$, the link breaks as two curves which approach (on some fundamental domain) the graphs of $\Z$-periodic functions as the perturbation goes to zero. 
The asymptotic  of the difference between these two functions is the \emph{Melnikov operator} $\cal M: W \in \Gamma^\infty_\leb(\R^3) \mapsto   \cal M(W)\in C^\infty (\R/\Z, \R)$.

Actually, $\cal M(W)$ depends only on $W|_\Gamma$: thus, passing to the quotient, we obtain an operator  $\cal M:  C^\infty(\Gamma,\R^3) \mapsto    C^\infty (\R/\Z, \R)$ still denoted by $\cal M$.

We will show in \cref{subsec:Meln:inverse}  that the Melnikov operator has a right  inverse:
\begin{proposition}\label{prop:Meln:smooth}
There exists a linear and continuous operator
	\[
	\cal I :g\in C^\infty(\R/\Z,\R)\to \cal I(g)\in C^\infty(\Gamma,\R^3)
	\]
	such that $\cal M(\cal I(g))=g$ and $\cal I(g)$ is compactly supported.
\end{proposition}
To prove this proposition, we will identify $\Gamma$ with the annulus $\R/\Z\times \R$, we will consider a bump function $\psi$ defined on $\Gamma$ such  that the integral of $\psi$ along each $X$-orbit in $\Gamma$  equals $1$.  Then  for $g\in C^\infty(\R/\Z,\R)$, with $\pi: \R\to \R/\Z$ 
and $p_2: \Gamma\equiv   \R/\Z\times  \R\to \R $ the canonical projections, we will prove that
$\cal I(g):= \psi\cdot g\circ \pi \circ p_2$  satisfies the sought properties. The idea is inspired by \cite{BerTur}.




\subsubsection{Step 2: Transforming a normal datum on a cylinder into a Beltrami vector field}
We would like to  complete (an approximation of) the normal vector field obtained in the first step into a Beltrami field on $\R^3$.  To this end we use \cref{main thm}  to obtain  the following:

\begin{theorem}\label{thm:Meln:final}
Let $U\subset \R^3$ be an open subset of $\R^3$ and let $X\in\mathscr{B}(U)$ display a strong heteroclinic link $\Gamma\subset U$. Let $\Sigma$ be a Poincar\'e section endowed with coordinates, chosen as for \cref{prop:Meln:smooth}. This defines the Melnikov operator $\cal M$. Then for every  nonempty open set $\cal U \in C^\infty(\R/\Z,\R)$  there exists a Beltrami field $W\in\mathscr{B}(U)$ whose Melnikov function $\cal M(W)$ belongs to $\cal U$. \end{theorem}
\begin{proof}
Let $g\in \cal U$.  By \cref{prop:Meln:smooth}, there is $\cal I(g)\in C^\infty(\Gamma,\R^3)$ such that $\cal M(\cal I(g))=g$.
  Applying \cref{main thm} to $\cal I(g)$,  we obtain a Beltrami field $W$ whose restriction  $W|_\Gamma$ is arbitrarily close to $\cal I(g)$. By continuity of $\cal M$, we obtain that
  $\cal M(W)\equiv \cal M(W|_\Gamma)$ is close to $g$ and so in $\cal U$.
\end{proof}

\subsubsection{Step 3: Prescribing jets of Melnikov functions on a double strong heteroclinic link}\label{step 3 preuve}
Now we start with a Beltrami field $X\in\mathscr{B}(U)$, for some bounded open set $U\subset\R^3$, displaying a \emph{double} strong heteroclinic link $(\Gamma^+,\Gamma^-)$, with $\Gamma^+,\Gamma^-\subset U$. 
Let $\cal M_+$ and $\cal M_-$ be the Melnikov operators, associated, respectively, to $\Gamma^+$ and $\Gamma^-$.

Our aim is to obtain a  Beltrami field unfolding on $\mathbb R^3$ which breaks $\Gamma^+$ to a transverse heteroclinic intersection and breaks  $\Gamma^-$  to a
 non-degenerate  unfolding of a quadratic  heteroclinic tangency.

 To do so, we will control the jets of both Melnikov functions $\cal M_+$ and $\cal M_-$. \cref{thm:Meln:final} can give, separately, when applied to $\Gamma^+$, a Beltrami field $X_+$ whose Melnikov function $\cal M_+(X_+)$ corresponds to a transverse heteroclinic intersection, and, when applied to $\Gamma^-$, a Beltrami field $X_-$ whose Melnikov function $\cal M_-(X_-)$ corresponds to a (quadratic) heteroclinic tangency. The difficulty is that we cannot just consider the Beltrami field $X_++X_-$, because we do not have any control on how $X_+$ and $X_-$ act, respectively, on $\Gamma^-$ and $\Gamma^+$.
 Indeed the non-disjointedness of  $\Gamma^+$ and $\Gamma^-$  forbids the application of \cref{GATcoro}.

We overcome this difficulty by obtaining  a \emph{non-degenerate unfolding of any multiplicity} of this double strong heteroclinic link.
In order to state this precisely, for  $g\in C^\infty(\R/\Z,\R)$, $\alpha\in\R/\Z$ and  $k\geq 0$, we identify the $k$-jet of the function $g$ at $\alpha$ with a vector in $\R^{k+1}$:
	\[J^k_{\alpha}(g):=(a_0,a_1,\dots,a_k) \in \R^{k+1} \quad \text{with}\quad a_i= D^i_\alpha g \; , \quad \forall 0\le i\le k\; . \]
	
\begin{proposition}\label{prop:transverse+k-het-tg}
	Fix $k\geq 0$ and let $\alpha \in\R/\Z$. There exist $\beta\in\R/\Z$ and  a $(k+3)$-dimensional space $\cal W$ of 
	Beltrami vector fields defined on $U$ such that the following map	is an isomorphism:
	\begin{equation} \label{iso}
	 W\in \cal W\mapsto \left(J^1_{\alpha }(\cal M_+(W)), J^k_{\beta}(\cal M_-(W))\right)\in \R^{k+3}\; .
	\end{equation}

\end{proposition}
The proof of this proposition will be presented in \cref{sec:LinALg}. It uses \cref{thm:Meln:final} and a new trick which allows us to avoid a computer-assisted estimate as in {\color{red} \cite{EPSR}}.
 The trick (see \cref{proof of prop:transverse+k-het-tg}) consists of considering three distinct points $\theta_1$,  $\theta_2$, $\theta_3$ in $\R/\Z$. Then by  \cref{thm:Meln:final}, the range of  $  J^1_{\alpha } \cal M_+|_{ \mathscr{B}(U)}$
 is 2-dimensional. We will show  using again \cref{thm:Meln:final} that the range of each map  $J^k_{\theta_i } \cal M_-|_{\mathscr{B}(U)}$ is $(k+1)$-dimensional and that  their product  has a range which is $(3k+3)$-dimensional. Equivalently, the kernel of $J^1_{\alpha } \cal M_+|_{\mathscr{B}(U)}$ is two-codimensional while the kernels of the
 $(J^k_{\theta_i } \cal M_-|_{\mathscr{B}(U)})_{1\le i\le 3}$ are $(k+1)$-codimensional and in general position.
By an elementary algebraic argument, this will imply that   the two-codimensional kernel of $J^1_\alpha\cal M_+|_{\mathscr{B}(U)}$ must be in general position with one of the 3 latter kernel of $ J^k_{\theta_i} \cal M_- |_{\mathscr{B}(U)}$, and so that  the product
$J^1_\alpha\cal M_ \times J^k_{\theta_i} \cal M_-|_{\mathscr{B}(U)} $ is onto. This implies immediately \cref{prop:transverse+k-het-tg}.
%
%


 \subsubsection{Step 4: Existence of homoclinic tangency unfolding non-degenerately in $\mathscr{B}(\R^3)$}
By the standard Poincar\'e-Melnikov Theorem (see \cref{sec:Melnikov}), the Melnikov operators $\cal M_+$ and $\cal M_-$ are the first order approximations of the displacement operators $\mathrm{displ}_+$ and $\mathrm{displ}_-$ respectively. For $\pm\in\{-,+\}$, each displacement operator $\mathrm{displ}_\pm$ associates to a sufficiently small perturbation of $X$ the distance function between the perturbed local stable and unstable manifolds at the heteroclinic link $\Gamma^\pm$, defined over some fundamental domain.

Together with \cref{prop:transverse+k-het-tg}, we then deduce that there exists a arbitrarily small neighborhood $\cal N$ of $0$ in $\Gamma^\infty_\leb(\R^3)$ such that the map
 $
 \cal N\cap\cal W\to \left( J^1_\alpha(\mathrm{displ}_+(W)), J^k_\beta(\mathrm{displ}_-(W)) \right)\in \R^{k+3}
 $  is a diffeomorphism onto  a neighborhood of $0$.

 For $k=2$, we can then select a family of Beltrami fields $(X_p=X+W_p)_{p\in\I}$ with $W_p\in\cal N\cap \cal W$ so that each $X_p\in\mathscr{B}(U)$ breaks $\Gamma^+$ to a transverse heteroclinic intersection 
 and the family $(X_p)_{p\in\I}$ of Beltrami fields in $\mathscr{B}(U)$ unfolds   $\Gamma^-$ to a quadratic heteroclinic tangency at some parameter $p_0$.

Thanks to the transverse heteroclinic intersection, we can use the para-inclination Lemma (see \cite{Berger2016}) to study the displacement operator between local stable and unstable manifolds of the same saddle point. It will imply that any sufficiently small $2$-jet of the displacement function will still be realised at some orbit. Thus, there exists a family of Beltrami vector fields $(X+\tilde W_p)_{p\in\I}$ in $\mathscr{B}(U)$, with $\tilde W_p\in\cal N\cap\cal W$ and the family $(\tilde W_p)_{p\in\I}$ $C^\infty$-close to the family $(W_p)_{p\in\I}$, unfolding non-degenerately a homoclinic quadratic tangency at a parameter $p_1$ close to $p_0$.
Now put $Y:=  X+\tilde W_{p_1}$ and $Z:= \partial_p \tilde W_{p_1}$. Since unfolding non-degenerately a quadratic homoclinic tangency is a $C^2$-robust condition on families, the family  $(Y+p\cdot Z)_{p\in \I}$ unfolds non-degenerately a homoclinic quadratic tangency at $p=0$.

Up to shrinking $U$ if necessary, we can assume it is bounded and such that $\R^3\setminus U$ is connected. Then the Global Approximation Theorem \ref{GAT} asserts the existence of fields $\tilde Y$ and $\tilde Z$ in
$\mathscr{B}(\R^3)$ such that their restrictions on $U$ are close to $Y$ and $Z$, respectively.
Since unfolding non-degenerately a quadratic homoclinic tangency is a robust condition, we conclude that $(\tilde Y+p\cdot \tilde Z)_{p\in\I}$ is the sought family of Beltrami fields in $\mathscr{B}(\R^3)$.

%


\cn

\subsection{Structure of the proofs of Theorems \ref{ThmB} and \ref{ThmC}: existence and density of  (para)-universal dynamics}
The strategy to prove Theorems \ref{ThmB} and \ref{ThmC} is to find in any non-empty open subset $\mathcal U \subset \cal N_{\mathscr B}(\R^3)$ and for any $k\ge 1$ a family $(X_p)_{p\in \I^k}$ of Beltrami vector fields $X_p\in \mathcal U$  which unfolds non-degenerately $k$ orbits of quadratic homoclinic tangencies for $k$ saddle cycles which are homoclinically related. Then we will restate the seminal Gonchenko-Shilnikov-Turaev theory \cite{Tu14,GST07} to show   the  existence and density of   universal and para-universal dynamics, which are the statements of  Theorems \ref{ThmB} and \ref{ThmC}.

The main technical novelty of the proof of Theorems \ref{ThmB} and \ref{ThmC} is to show  the existence of  a family $(X_p)_{p\in \I^k}$ of Beltrami vector fields $X_p\in \mathcal{U}$ which unfolds non-degenerately $k$ orbits of quadratic homoclinic tangencies for $k$ saddle cycles which are homoclinically related. This will be done by exploring the beautiful work of Duarte to obtain homoclinically related but disjoint wild hyperbolic sets in order to  apply  perturbative \cref{GATcoro}.

Before giving more details on these proofs,  let us now formalize the statements to be proved,  by   recalling standard definitions in bifurcation theory of  surface diffeomorphisms or 3-dimensional flows.

\begin{definition}\label{multi homo}
Let $J\ge 1$ and 
 let  $(f_p)_{p\in \I^J}$ be a family  of surface diffeomorphisms. Assume that $f_0$ has a saddle periodic orbit $O$ which displays  $J$ different quadratic homoclinic tangencies at $q_j\in W^s(O; f_0)\cap W^u(O; f_0)$ for $1\le j\le J$.  This \emph{$J$-tuple of  homoclinic tangencies unfolds non-degenerately} at $p=0$ if, with $\mu_j(p)$ a  relative position  between $W^s(O; f_p)$ and $W^u(O; f_p)$ at $q_j$,  the following map is a local diffeomorphism at $p=0$:
\[p\in \I^J\mapsto (\mu_j(p))_{1\le j\le J}\in \R^J\; .\]
\end{definition}
Let us state the vector field counterpart of Definition \ref{multi homo}.
\begin{definition}
A \emph{family $(X_p)_{p\in \I^J}$ of vector fields displays  a $J$-tuple of  homoclinic tangencies that unfolds non-degenerately at $p=0$}  if there is a Poincar\'e section $\Sigma$ for $X_0$ which defines
a family  of surface diffeomorphisms $(f_p)_{p\in (-\epsilon, \epsilon)^J}$
which   displays  a saddle point with a $J$-tuple of   homoclinic tangencies that  unfolds non-degenerately at $p=0$.
\end{definition}

We shall show that the proof \cite{GST07} on universal conservative $C^r$-surface dynamics actually applies to the following broader setting:
\begin{definition}
Let $\cal F$ be a Fr\' echet manifold formed by smooth symplectic surface diffeomorphisms or smooth conservative flows of a 3-manifold. The space  $\cal F$ is \emph{GST-wild} if the following property holds true. For every $J\ge 1$ and every non-empty open subset $\mathscr U$ of $\cal F$,
there  is  a smooth parameter family   of  dynamics in $\mathscr U$
parametrized by $\I^J$  which  displays  a $J$-tuple of quadratic  homoclinic tangencies that unfolds non-degenerately at $p=0$.
\end{definition}

It follows from classical techniques \cite{Ku63,Sm63,BT86} that  the Newhouse domains of $C^\infty_\leb$ or $C^\omega_\leb$-surface dynamics  and the ones of $ C^\infty_{\leb}$ or
$ C^\omega_{\leb}$-flows on $\R^3$  are GST-wild.

The Gonchenko-Shilnikov-Turaev theorem states that universality is generic in GST wild spaces and that they are para-universal. See also \cite{GelTur10}.
These notions of universality are defined   in the setting of vector fields in \cref{def universal,def para universal} using the notion of   renormalized iteration defined in  \ref{def renormalized ietration}.  These are canonically translated into the discrete setting:
\begin{definition}
A \emph{renormalized iteration} $G$ of a $C^\infty_\leb$-diffeomorphism $f$ of a surface $M$  equals to  a return map   at some disk of $M$  in rescaled coordinates. More precisely, there are nested disks $\check \Sigma \Subset \Sigma\subset M$, an integer $N\ge 1$ and coordinates $\phi: \R^2\to   \Sigma$ sending $\D$ onto $\check \Sigma$ such that $f^N(\check \Sigma) \subset \Sigma$ and $G= \phi^{-1} \circ f^N\circ \phi|_{\D}$ is symplectic.

\end{definition}

\begin{definition}[\cite{BD00,Tu03,GST07}]\label{def universal discrete}
A $C^\infty_\leb$-diffeomorphism $f$ of a surface $M$ is \emph{universal} if the set of renormalized iterations of $f$ is dense in $\Diff^\infty_{\leb+}(\D,\R^2)$. The notions of compactly universal and para-universal are defined similarly.
\end{definition}
\begin{remark}
A surface map $f$ is (resp. compactly) universal iff its suspension defines a vector field which is (resp. compactly) universal.
\end{remark}

We will extract from part of the proof of   \cite{Tu03,GST07} the following abstract result:
\begin{theorem}[Gonchenko-Shilnikov-Turaev]\label{TGST}
Let $\cal F$ be a Fr\' echet submanifold  of the space of smooth
conservative vector fields or the space of smooth symplectic surface diffeomorphisms
 which is GST wild.
 Then there exists a  topologically generic subset  $\cal G\subset \cal F$   formed by compactly  universal dynamics.
   Moreover $\cal F$ is para-universal.
\end{theorem}
This statement did not appear written like this in Gonchenko-Shilnikov-Turaev works, but it can be deduced from a simple combination of their lemmas following a general scheme they use. This scheme consists of extracting from a $k$-tuple of quadratic homoclinic tangencies a homoclinic tangency of multiplicity $k$ which unfolds non-degenerately \cite{GST07}. Then from this non-degenerate unfolding of homoclinic tangency of multiplicity $k$, a family of renormalizations equal to the full family of H\'enon maps of degree $k$ emerges \cite{GST07}. Finally they use Turaev's Theorem which asserts that any map or family of maps in $\Diff^\infty_\leb(\D, \R^2)$  can be approximated by a renormalization of an iterate of a H\'enon map of some large degree \cite{Tu03}. This beautiful scheme will be detailed in \cref{proof TGST}.  \medskip

The main issue of the proof of Theorems \ref{ThmB} and \ref{ThmC} is to reach the setting of the Gonchenko-Shilnikov-Turaev theory using only  perturbative  \cref{GATcoro}, as given by the following:
	\begin{theo}\label{ThmD}
The open set  $\cal N_{\mathscr B}(\R^3)$ in  ${\mathscr B}(\R^3)$ is GST wild.
%
\end{theo}
To show this theorem, we fix a nonempty open subset $\mathscr U$ of
 $\cal N_{\mathscr B}(\R^3)$.  By definition of $\cal N_{\mathscr B}(\R^3)$, there are $X_0\in \mathscr U$ and a Beltrami field $W_1$ such that
   the family $(X_p=X_0+p\cdot W_1)_{p\in \mathbb{I}}$   unfolds a quadratic homoclinic tangency at $p=0$.  Then we will explore  the powerful   work of Duarte  \cite{Du99,Du} to construct a
 hyperbolic basic set with a robust homoclinic tangency. From this we will extract a sub-parameter family which displays a homoclinically related
but disjoint   hyperbolic set displaying a robust homoclinic tangency.  Using the fact that these hyperbolic sets are disjoint, we will cutoff  the perturbation  $W_1 $ using perturbative \cref{GATcoro} to create a new Beltrami field $W_2$ on $\R^3$. This will create a two parameter family $(X^1_p=X_0+p_1\cdot W_1+p_2\cdot W_2)_{p=(p_1,p_2)\in \I^2}$ unfolding non-degenerately two quadratic homoclinic tangencies at a parameter $p$ close to $0$. We will iterate inductively this construction to obtain the sought $k$-parameter families. The proof will be done in \cref{proof thm D}.

Now let us remark that \cref{TGST,ThmD} immediately imply Theorems \ref{ThmB} and \ref{ThmC}:
\begin{proof}[Proof of Theorems \ref{ThmB} and \ref{ThmC}]
By \cref{ThmD}, the Baire space $\cal N_{\mathscr B}(\R^3)$
is GST wild, and so it satisfies  the assumptions of \cref{TGST}.
Thus a topologically generic  subset of $\cal N_\mathscr{B}(\R^3)$ is formed by compactly universal vector fields   and  $\cal N_{\mathscr B}(\R^3)$ is compactly  para-universal as stated in  Theorems \ref{ThmB} and \ref{ThmC}.
\end{proof}
\section{Elements of bifurcation theory}
\subsection{Involved topologies}\label{def of topology}
Let us precise the topologies involved in the statements of the main results.

For every $k\ge 0$, let $B_k\subset\R^3$ be the ball centered at zero of radius $k$.
We recall that $ \Gamma^\infty_{\leb}(\R^3)$ endowed with the family $(N_k)_k$ of semi-norms $N_k: X\mapsto  \Vert X|_{B_k} \Vert_{C^k}$ is a Fr\'echet space.   Equivalently, the topology is defined by the following complete distance:
	\[
	d_{C^{\infty}}: (X,Y)\in \Gamma^\infty_{\leb}(\R^3)^2\mapsto \sum_{k=0}^\infty\dfrac{1}{2^{k}}\min \left(   N_k(X-Y)
	, 1 \right)\, .
	\]
Note that the subspace of Beltrami fields is closed in $ \Gamma^\infty_{\leb}(\R^3)$, but the subspace of Beltrami fields with sharp decay is not closed.

Hence we endow $ \mathscr B(\R^3)$ with the following
family $(\tilde N_\alpha)_{\alpha\in \N^3}$ of norms:
\[
\tilde N_\alpha: X\mapsto  \sup_{x\in \R^3}  (1+\|x\|)\cdot \|D^\alpha X(x)\|
\] with which it is a Fr\'echet space. Equivalently the topology is defined by the following complete distance:
	\[
		d_{\mathscr B}: (X ,Y )\in \mathscr B(\R^3)^2\mapsto \sum_{\alpha\in \N^3}
2^{-|\alpha|}\min \left( \tilde N_\alpha(X-Y)
,   1 \right)
\;
\]

\subsection{The Gonchenko-Shilnikov-Turaev theory}
\label{proof TGST}

\subsubsection{Constructing non-degenerate unfoldings of homoclinic tangencies of any  multiplicity}
A key configuration in the Gonchenko-Shilnikov-Turaev's theory is that of homoclinic tangencies of higher multiplicity. Let us state this notion in the setting of surface diffeomorphism; using Poincar\'e return maps we will translate it into the context of 3-dimensional flows at the end.
\begin{definition}\label{hom tg higher mult} For $k\ge 1$, a  saddle periodic point $P$ of a surface  diffeomorphism $f$ displays a \emph{homoclinic tangency of multiplicity $k$}, if the local stable and unstable manifolds of $P$ are tangent at a point $Q$ and the order of contact of the tangency is $k$. More precisely this means that there is a chart of a  neighborhood $N$  of $Q$ which  identifies  $Q\equiv 0$,
 $W^u_{loc} (P)\cap N\equiv\I\times \{0\}$ and $W^s_{loc} (P)\cap N$ with the graph of a function $w: \I\to \R$ such that:
\[0=w(0)=Dw(0)= \cdots= D^k w(0)\qand  D^{k+1} w(0)=1\; .\] \end{definition}

\begin{example} We notice that a quadratic homoclinic tangency is a homoclinic tangency of multiplicity $1$. A cubic homoclinic tangency is a homoclinic tangency of multiplicity $2$.\end{example}

Now let $(f_p)_{p\in \I^k}$ be an unfolding of  a surface  diffeomorphism $f=f_0$  displaying a homoclinic tangency of multiplicity $k\ge 1$ for a saddle point $P$. Then the periodic point $P$ persists for $p$ small, and its local stable and unstable manifolds as well. This induces an unfolding of the homoclinic tangency.
\begin{definition} The unfolding of the homoclinic tangency of multiplicity $k$  is  \emph{non-degenerate} if  there are parameter dependent coordinates of a neighborhood $N$ of $Q$
which identify  $Q\equiv 0$, $ W^u_{loc} (P; f_p)\cap N\equiv  \I\times \{0\}$ and  $W^s_{loc} (P; f_p)\cap N$ with $\{(x, w_p(x)): x\in \I\}$ where $(w_p)_{p\in\I^{k}}$  satisfies that:
\begin{itemize}
	\item it holds  $w_0(0)=Dw_0(0)= \cdots= D^k w_0(0)=0$  and $D^{k+1}w_0(0)\neq 0 $,
	\item the   map  $ p\in \I^k\mapsto (D^j w_p(0))_{0\le j\le k-1}\in \R^k $ is a diffeomorphism onto its image.
\end{itemize} \end{definition}

When the multiplicity of the tangency is 1, or equivalently when the homoclinic tangency is quadratic, saying that the unfolding is non-degenerate means roughly speaking that  the \emph{relative position} $p\mapsto w_p(0)$ of the continuation of the  stable and unstable manifolds at the tangency point has non-zero derivative at $p=0$.

The following is the first Theorem of the GST-theory. It enables to transform a non-degenerate unfolding of $J$-quadratic homoclinic tangencies into a non-degenerate unfolding of a homoclinic tangency of multiplicity $J$:
\begin{theorem}\label{Thmlem5GST07}
Let $(f_p)_{p\in \I^J}$ be a family of surface diffeomorphisms which unfolds non-degenerately at $p=0$ a $J$-tuple of (different)  quadratic homoclinic tangencies of a saddle orbit $O$.
Then there is $p_1\in \I^J$ arbitrarily small such that at this parameter, the (hyperbolic continuation of the) saddle  $O$ displays a homoclinic tangency of multiplicity $J$. Moreover this homoclinic tangency of multiplicity $J$  is non-degenerately unfolded by the family $(f_p)_{p\in\I^J}$.
%
%
\end{theorem}
\begin{proof}
We actually apply inductively on $2\le k\le J$ the following lemma:
\begin{lemma}[{\cite[Lem. 5 ]{GST07}}] \label{lem5GST07}
Let $k\ge 2$. Let $(f_p)_{p\in \I^k}$ be a $C^\infty$-family  of symplectic  surface diffeomorphisms such that   $f_0$ has a  saddle orbit  $O$ which displays  a homoclinic   tangency of multiplicity $k-1$ and a quadratic homoclinic tangency.
Assume that the first unfolds non-degenerately in the family $(f_p)_{p\in \I^{k-1}\times \{0\}}$ while the second unfolds non-degenerately in the family $(f_p)_{p\in \{0\}\times \I}$. Then there is $\tilde p\in \I^{k}$ arbitrarily  small, such that $O$ displays a  homoclinic  tangency of multiplicity $k$   at $p=\tilde p$ which unfolds non-degenerately in $(f_p)_{p\in \I^k}$.
\end{lemma}
The mechanism of the proof of \cref{lem5GST07} is clearly illustrated in Figure 11 in \cite{GST07}.
\end{proof}

\subsubsection{H\'enon maps as renormalized iterates nearby homoclinic unfoldings}
The next key ingredient of the Gonchenko-Shilnikov-Turaev theory involves the conservative  H\'enon family.

\begin{definition}A conservative H\'enon map of degree $ d$  is a map of the form:
\[H_g:= (x,y)\mapsto (g(x) -y,x);\]
with $g$ a polynomial map of degree $  d$. If $\R_k[X]$ denotes the space of polynomial maps of degree $d\le k$, the  full family of conservative H\'enon map is $(H_g)_{g\in \R_k[X]}$.
\end{definition}
The second theorem of their theory  states that nearby a non-degenerate unfolding of a homoclinic tangency of multiplicity $k$, there are renormalizations close to the family of conservative H\'enon maps of degree $k-1$:
\begin{theorem} \label{Lemma6GST bis}
Let 	$(f_q)_{q\in \I^k}$ be a smooth family of symplectic  diffeomorphisms  of a surface $S$ which  unfolds  non-degenerately a homoclinic tangency of multiplicity $k$ at $q=0$. For every $M>0$, let:
\[\B_k(M):= \{g (X)=  \sum_{i=0}^{k-1} a_i \cdot X^{i}\in \R[X]: a_i\in [-M,M], \;  \forall 0\le i \le k-1\}\]
Then  for every $\epsilon>0$,   there exist:
\begin{itemize}
\item  an embedding $Q: \B_k(M)\hookrightarrow [-\epsilon, \epsilon]^k$ and  an integer $N>0$,
\item  a smooth family $(C_p)_{p\in \B_k(M)}$ of charts $C_p: (-M,M)^2\hookrightarrow  S $ of constant determinant onto a small open set nearby the tangency  point,
\end{itemize}
 such that  $(H_{g}|_{(-M,M)^2} )_{g\in  \B_k(M)}$
 is $\epsilon$-$C^\infty$-close to
 the family   $( \cal R f_g)_{g\in \B_k(M)}$ formed by
 $\cal R f_g:=  C_g^{-1} \circ f^N_{Q(g)}\circ C_g$.  	\end{theorem}
\begin{proof} This statement is almost the same as Lemma 6 in \cite{GST07}, the only difference is that we are dealing with polynomials which do not need to be unitary. Nevertheless, using a rescaling of the phase space by a large factor  $\Lambda>0$, $(x,y)=(\Lambda^{-1} X,\Lambda^{-1} Y)$, the theorem can be immediately deduced from this lemma with the following setting:
	\begin{lemma}[Lemma 6 \cite{GST07}] \label{Lemma6GST}  Under the assumptions of
	\cref{Lemma6GST bis},  for every $M,\Lambda,\epsilon >0$, there are:
\begin{itemize}
\item  an embedding $Q: \I^k\hookrightarrow [-\epsilon, \epsilon]^k$ and  an integer $N>0$,
\item   a smooth family $(C_p)_{p\in \B_k(M)}$ of charts $C_p: (-\Lambda^{-1} M,\Lambda^{-1} M)^2\hookrightarrow S $ of constant determinant onto a small open set nearby the tangency  point,
\end{itemize}
 such that  $\cal R f_p:=  C_p^{-1} \circ f^N_{Q(p)}\circ C_p$
  forms a family   $( \cal R f_p)_{p\in \I^k}$ which is $\epsilon$-$C^\infty$-close to \[(H_{g_p}|_{(-\Lambda^{-1} M,\Lambda^{-1} M)^2} )_{p\in  \I^k}\] with $g_p (x)\mapsto  \sum_{i=0}^{k-1} M\cdot\Lambda^{i-1}\cdot  p_i \cdot x^{i} +x^k$ and $p=(p_0,\dots, p_{k-1})$.
	\end{lemma}

	\end{proof}
\subsubsection{Proof of \cref{TGST}}

The final result of this theory is due to	Turaev. It says that renormalizations of H\'enon maps are dense in $\Diff^\infty_{\Leb+}(\D, \R^2)$ and likewise for the families.
	\begin{theorem}[\cite{Tu03}, Lemma 9 \cite{GST07}] \label{Lemma9GST}
		For every   $J\ge 0$ and  every non-empty open subset $\cal O_J\subset  \Diff^\infty_{\Leb+}(\D, \R^2)_{\I^J}$,    there exist $k>0$, a $C^\infty_{\Leb+} $-family of polynomials $(g_p)_{p\in \I^J}$ of  degree 	$< k$, 	an iteration $N\ge 1$ and an affine change of coordinates  $\phi$ of  $\R^2$   such that $(\phi^{-1}\circ H_{g_p}^N\circ \phi|_\D )_{p\in \I^J}$ is in $\cal O_J$, where $H_{g_p}$ is   the H\' enon-like map		$(x, y)\mapsto (g_p(x)-y, x)$.
	\end{theorem}

We are now ready for:
\begin{proof}[Proof of \cref{TGST}]
Let $\mathscr E$ be a Baire subset of $\Gamma^\infty_\leb (\R^3)$ which is GST wild. By \cref{para universal are universal}, it suffices to show that $\mathscr E$ is para-universal. Hence it suffices to prove that for any $J\ge 0$, and any non-empty subsets $\mathscr U\subset \mathscr E$ and $\mathscr O_J \subset \Diff^\infty_\Leb(\D, \R^2)_{\I^J}$, there exists a smooth $J$-parameter family in $\mathscr U$ which has a renormalization in $\mathscr O_J$.
By \cref{Lemma9GST}, there exist a $k\ge 0$ and a smooth $J$-parameter family of conservative H\'enon maps $(H_{g_p})_{p\in \I^J}$ which displays a renormalization in $\mathscr O_J$.
By \cref{Thmlem5GST07} and the definition of the GST wild property, there exists a $k$-parameter family of dynamics in $\mathscr U$ which unfolds non-degenerately a homoclinic tangency of multiplicity $k$ of a saddle point. By \cref{Lemma6GST bis}, we can extract from this $k$-parameter  family a $J$-parameter subfamily which is arbitrarily close to $(H_{g_p})_{p\in \I^J}$ and so displays a second renormalization in $\mathscr O_J$.
As a renormalization of a renormalization is a renormalization, we obtained the sought result.
\end{proof}

\subsection{Geometry of wild hyperbolic sets after Newhouse and Duarte}
We recall that a \emph{hyperbolic basic set} for a diffeomorphism is an invariant compact set $\Lambda$ which is  transitive, hyperbolic and
locally maximal (it is the maximal invariant set in one of its neighborhoods $\mathscr U$).
This generalizes the notion of saddle periodic orbits, which are finite sets, to compact sets.
 All the points $z\in \Lambda$ have local stable and unstable manifolds, which are injectively  immersed submanifolds  depending continuously on $z$ in the $C^\infty$-topology.  The stable and unstable manifolds of the orbit of $z$ are the union of stable and unstable manifolds of the iterate of $z$; these are also  injectively immersed submanifolds.

 By Anosov's Theorem, hyperbolic basic sets are structurally stable: for every smooth (actually $C^1$ suffices) perturbation of the dynamics, the maximal invariant set $\tilde \Lambda$  in $\mathscr U$ for the perturbed dynamics is a hyperbolic basic set, and the dynamics on $\tilde \Lambda$  is \emph{orbit equivalent} to the one on $\Lambda$. There is a homeomorphism $h: \Lambda\to \tilde \Lambda$ close to the identity which sends each orbit in  $\Lambda$ to an orbit in $\tilde \Lambda$ called its \emph{hyperbolic continuation}.

We say that  two hyperbolic basic sets $\Lambda_1$ and $\Lambda_2$ are \emph{homoclinically related} if there are $z_1\in \Lambda_1$ and $z_2\in \Lambda_2$ such that the stable manifold of $z_i$ intersects transversally the unstable manifold of $z_j$ for $1\le i\neq j\le 2$.

We will use many times the following generalization of the classical inclination Lemma:
\begin{lemma}[Para-inclination lemma  {\cite[Lemma 1.7]{Berger2016}}]\label{para-inclination lemma}
	Let $r\geq 1$ and let $U\Subset \R^m$. Suppose that $(f_p)_p$ is a smooth family of diffeomorphisms $f_p$ of $U$ which leaves invariant a basic hyperbolic set $\Lambda_p$. Let $z\in\Lambda_0$ and $(D_p)_p$ be a smooth family of $C^r$ manifolds of the same dimension as $W^u_{loc}(z;f_p)$ and such that $(D_p)_p$ intersects transversally $(W^s_{loc}(z;f_p))_p$ at a smooth curve of points $(y_p)_p$. Then, for every $\epsilon>0$ and for every $n$ large, the $\epsilon$-neighborhood of $f^n_p(y_p)$ in $f_p^n(D_p)$ is a submanifold $D^n_p$ that is $C^\infty$-close to $W^u_{\epsilon}(f^n_p(z);f_p)$ and $(D^n_p)_p$ is $C^\infty$-close to $(W^u_{loc}(f^n_p(z);f_p))_p$.
\end{lemma}
\cn

%

The concept of homoclinic tangency can be generalized to hyperbolic basic sets. We say that a hyperbolic basic set $\Lambda$ displays a \emph{homoclinic tangency} if there is an orbit $O\subset  \Lambda$ whose stable manifold   is tangent to its unstable manifold.

Newhouse discovered the first mechanism that yields \emph{wild}  hyperbolic basic sets:
\begin{definition}[Robust quadratic tangency]\label{def wild}
Let $f$ be a surface $C^2$-diffeomorphism leaving invariant a hyperbolic basic set $\Lambda$.
We say that $\Lambda$ is \emph{wild} if it  displays a \emph{robust quadratic homoclinic tangency}: there are continuous families $(W^s_{loc} (z; f))_{z\in \Lambda}$ and $(W^u_{loc} (z; f))_{z\in \Lambda}$  of local stable and unstable manifolds and a $C^2$-open neighborhood $  {\cal N}$  of $f$ satisfying the following property.
	
	For every $\tilde f\in {\cal N}$, the hyperbolic continuation $  \tilde \Lambda$ of $\Lambda$ is well defined and there exists a point $z\in \Lambda$ and an iteration $n\ge 0$ such that the continuation of the local stable and unstable manifolds
	$W^s_{loc} (z; \tilde f)$ and $W^u_{loc} (\tilde f^n(z); \tilde f)$ display a quadratic homoclinic tangency.

Let $(f_p)_{p\in \I}$ be a family containing $f=f_0$. We say that this family \emph{unfolds non-degenerately the homoclinic tangencies of $\Lambda$}, if it unfolds non-degenerately  each
tangency  between $W^s_{loc} (z; f_q)$ and $W^u_{loc} (z'; f_q)$, among $z,z'\in \Lambda$ and $q\in \I$ small.
\end{definition}

\begin{proposition} \label{robust implies homo}
Let $(f_p)_{p\in \I}$ be a family unfolding non-degenerately the quadratic  homoclinic  tangencies of a wild basic hyperbolic set $\Lambda$ at $p=0$. Let $P\in \Lambda$ be a saddle periodic point. Then there exists an arbitrarily small parameter at which the orbit of $P$ displays a  quadratic  homoclinic  tangency which unfolds non-degenerately.
\end{proposition}
\begin{proof}
	Let $z\in \Lambda$ and $n\ge 0$ be so that $W^s_{loc} (z;f_0)$ displays a quadratic tangency with $W^u_{loc} (f_0^n(z); f_0)$ which unfolds non-degenerately. Let $N\gg 1$.
	Let $O_P$ be the orbit of $P$. We recall that $W^s_{loc} (O_P; f_0)$ intersects transversally $W^u_{loc} (f^{N}_0(z); f_0)$ and $W^u_{loc} (O_P; f_0)$ intersects transversally $W^s_{loc} (f^{n-N}_0(z); f_0)$.
	As $N$ is large, by the inclination lemma $W^s (O_P; f_0)\supset f_0^{-N}(W^s_{loc} (O_P; f_0)) $ contains a segment which is close to  $W^s_{loc} (z; f_0)$ and  $W^u  (O_P; f_0)\supset f^N_0(W^u_{loc} (O_P; f_0))$ contains a  segment  which is close to $W^u_{loc} (f^n_0(z); f_0)$. By the para-inclination lemma \cite{Berger2016}, the hyperbolic continuations of these segments vary with $p$ $C^\infty$-close to the  hyperbolic continuations $(W^u_{loc} (f^n_p(z); f_p))_p$ and $(W^s_{loc} (z; f_p))_p$. Thus, as the unfolding is non-degenerate, there exists an arbitrarily small parameter such that a pair of segments in $W^s(O_P;f_p)$ and $W^u(O_P,f_p)$ display a quadratic tangency unfolded non-degenerately.\cn
\end{proof}

Newhouse \cite{Ne1} gave the first example of a horseshoe displaying a robust homoclinic tangency.  Then he proved  \cite{Ne2} that those sets appear at any non-degenerate unfolding of  a quadratic homoclinic tangency of a dissipative saddle point of a surface diffeomorphism.   The conservative counterpart of this result was proved by Duarte:

\begin{theorem}[Duarte {\cite[Theorem B]{Du}}]\label{Duarte thm}  Let $(f_p)_p$  be a smooth one-parameter family of symplectic maps in $\Diff^\infty_{\leb^+}(\Sigma', \Sigma)$. Let $P$ be a periodic hyperbolic orbit displaying a quadratic homoclinic tangency with orbit $Q$,   which unfolds non-degenerately at $p= 0$. Take any small neighborhood $U$ of $P\cup Q$. Then there is a sequence of parameters  $p_n $ converging to $p= 0$ for which $f_{p_n}$ leaves invariant a hyperbolic horseshoe\footnote{A hyperbolic horseshoe is a locally maximal, transitive hyperbolic Cantor set.}       $\Lambda_n\subset U$ which is  homoclinically related to the hyperbolic continuation of $P$ and   displays   a robust homoclinic tangency.   \end{theorem}
We now give a sketch of the proof of Duarte's result, because it will be needed to deduce \cref{Duarte coro} below.
\begin{proof}[Sketch of proof]  
The beautiful  proof of Duarte can be divided into three main steps.
	\begin{enumerate}[$(i)$]
		\item Refering to the work of  Gonchenko and Shilnikov in \cite{GonShi97} (see also \cite{HenonGST,Henon2GST}) and of Mora and Romero in \cite{MorRom97}, a H\'enon-like renormalization procedure can be done nearby one point $Q_0$ of the homoclinic tangency orbit $Q$ at $p=0$ in the conservative setting. This will be recalled below. The scheme has as limit the \emph{conservative H\'enon family}
		\[
		(a,x,y)\mapsto H_a(x,y):=(y,-x+a-y^2)\,.
		\]
		\item In the framework of the conservative H\'enon family, Duarte showed that there exists a parameter $a^*$ close to $-1$ 
		such that the H\'enon map $H_{a^*}$ displays a wild horseshoe $\Lambda_{a^*}$, i.e., a hyperbolic basic set with robust homoclinic tangency that unfolds non-degenerately.
		\item These first two steps are used to construct a sequence of parameters $(p_n)_n$ converging to $p=0$ 
		 such that, for every $n$, 
		 the diffeomorphism $f_{p_n}$ leaves invariant a wild horseshoe $\Lambda_n$. 
		 In \cite{Du99}, it is proved that this sequence $(\Lambda_n)_n$ converges to $P\cup Q$ in the Hausdorff metric and that each $\Lambda_n$ is homoclinically related to the hyperbolic continuation of $P$.
	\end{enumerate}
	
	We now give more details and references on each step of Duarte's proof.
	
	Concerning the renormalization procedure, for every large $n\geq 0$ we can consider a small box $I_n\times B_n$ in $\mathbb{I}\times\Sigma'$ near the point $(0,Q)$, which shrinks to the point itself as $n\to+\infty$ and so that the segments $(I_n)_n$ are disjoint. This small box is chosen so that it is sent sufficiently close to itself by the function $(p,X)\mapsto (p, f_p^n(X))$.
	In each small box we rescale the coordinates by a map
	\[
[-30, 30]\times 	[-4,4]^2\ni (a,x,y)\mapsto (\phi_n(a),\psi_{n,a}(x,y))\in I_n\times B_n
	\]
which enables to see  $f_{\phi_n(a)}^n$ as a map  $\mathcal{R}_nf_a:=\psi^{-1}_{n,a}\circ f_{\phi_n(a)}^n\circ\psi_{n,a}$ which is $C^\infty$ -close to each H\'enon map $H_a:(x,y)\mapsto (y,-x+a-y^2)$  when $n$ is large. In particular, the presence of a $C^2$-stable wild horseshoe for the H\'enon map $H_{a^*}$ for the parameter $a^*$ (that will be proved in the second step) will persist in the renormalized dynamics at the unfolding of the homoclinic tangency, i.e., for $f_{p_n}$ with $p_n:=\phi_n(a^*)$, for $n$ large.\\ 

In the second step we study the conservative H\'enon family of maps $(x,y)\mapsto H_a(x,y)=(y,-x+a-y^2)$. At $a=-1$, the map displays a parabolic fixed point which breaks, for $a>-1$, into a saddle fixed point $S$ and an elliptic fixed point $E$. For values of $a>-1$ close to $-1$, the point $S$ exhibits a transverse homoclinic intersection $\Omega$. The asymptotics of the homoclinic angle at $\Omega$ is given in \cite{GelSau}. All the points $S,E,\Omega$ belong to the axis of symmetry $\{x=y\}$. Moreover, the saddle point $S$ is accumulated by a sequence $(Q_n)_n$ of periodic points of period $2n$. 
	For each $n$, we can define $R_n$ as the rectangle bounded by the local stable and unstable manifolds of $S$ and $Q_n$. This is the starting point for Duarte's construction \cite[Pages 7-8]{Du} of a (positive binary) horsehoe, depicted in \cref{figDuarte}. More precisely, let $R_n^0\subset R_n$ be the rectangle formed by points whose first iteration remains in $R_n$, while $R_n^1\subset R_n$ is the rectangle of points which come back to $R_n$ after $2n$ iterations. Denoting by $T_n:R_n^0\cup R_n^1\to R_n$ the first return map to\footnote{More precisely, the map $T_n$ associates to a point in $R_n^0$, resp. in $R_n^1$, its first iterate, resp. its $2n$-th iterate.} $R_n$, our candidate horseshoe is then the maximal invariant set $\bigcap_{k\in\Z} T_n^k(R_n^0\cup R_n^1)$.
	\begin{figure}
		\centering
		\includegraphics[width = 13cm]{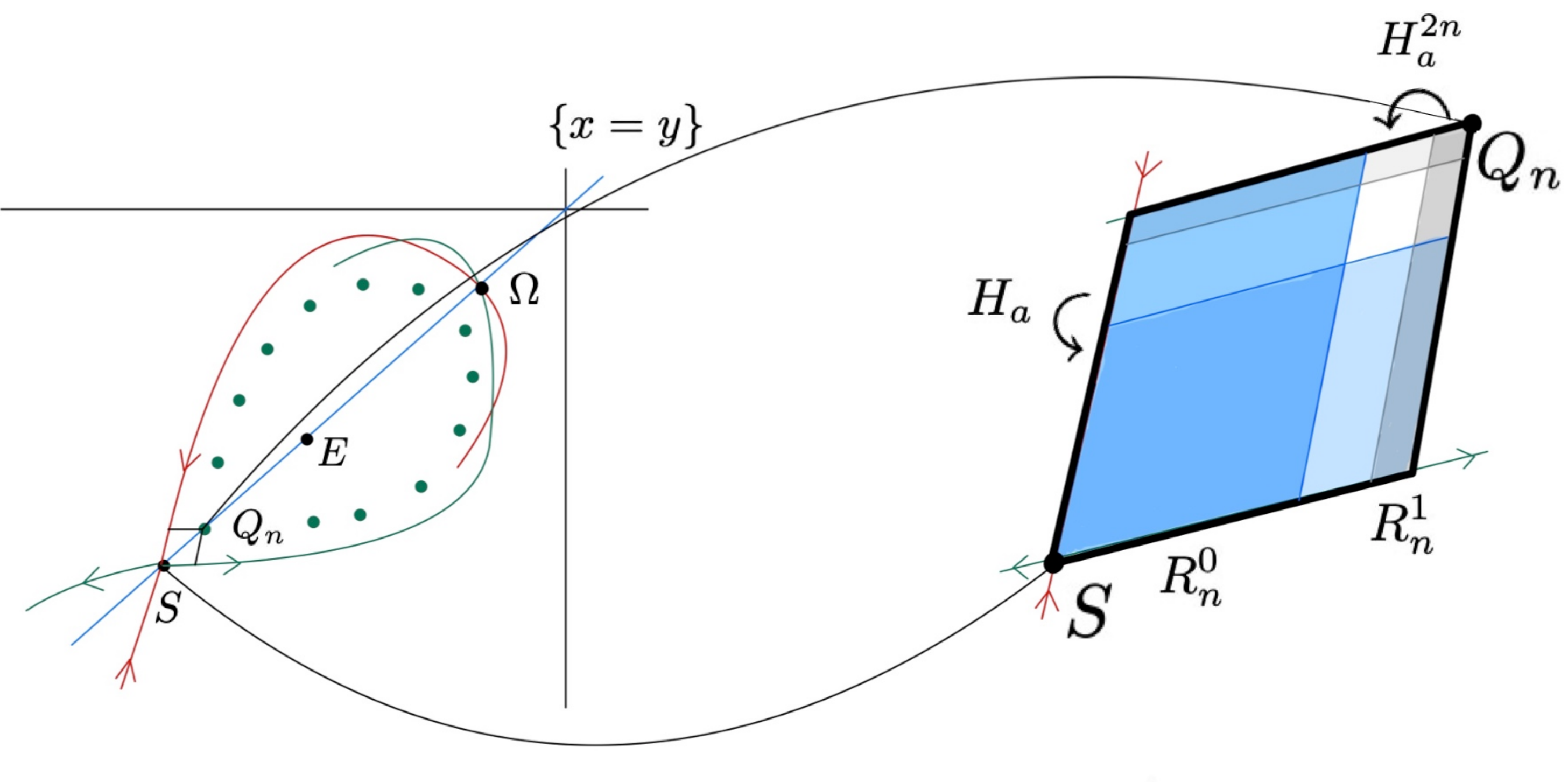}
		\caption{Construction of the horseshoe for the H\'enon map.
 }
		\label{figDuarte}
	\end{figure}
	The study of the structure of the aforementioned horseshoe (its thickness) 
	enables him to show that any eventual homoclinic tangency of the horseshoe will persist under perturbation, i.e., the horseshoe will be wild. See \cite[Section 5]{Du} for the proof. The last remaining key point of Duarte's proof is then showing the existence of quadratic homoclinic tangencies. 
	Duarte analyses the stable and unstable manifolds of $S$ (see \cite[Section 6]{Du})
	and shows that there is $a^*$ close to $-1$ for which the saddle point $S$ displays a quadratic homoclinic tangency that, moreover, unfolds non-degenerately when $a$ varies, see \cite[Lemma B]{Du}.\\
	
	In the second step Duarte proved the existence of a wild horseshoe $\check\Lambda$ left invariant by the H\'enon map $H_{a^*}$. As $\mathcal{R}_nf_{a^*}$ is $C^2$-close to $H_{a^*}$ when $n$ is large, the continuation of $\check\Lambda$ remains wild for the renormalization $\mathcal{R}_nf_{a^*}$ of $f_{p_n}$. Thus the $f_{p_n}$-orbit $\Lambda_n$ of $\psi_{n,a^*}(\check\Lambda)$ is a wild horseshoe for $f_{p_n}$. 
	In \cite[Section 6]{Du99}, Duarte detailed the proof that, for any neighborhood $U$ of $P\cup Q$, the horseshoe $\Lambda_n$ is contained in $U$ for $n$ sufficiently large. Moreover, in \cite[Lemma 7.1]{Du99}, he showed that each horseshoe $\Lambda_n$ is homoclinically related to the hyperbolic continuation of $P$. 
\end{proof}

From items $(i)$ and $(ii)$ of Duarte's proof, we have the following corollary:
\begin{coro}\label{Duarte coro}
Under the assumption of \cref{Duarte thm}, the unfolding of $\Lambda_n$ in $(f_p)_p$ at $p=p_0$ is non-degenerate. Moreover, for every point $Q_0$ in the homoclinic orbit $Q$ and for every neighborhood $U_0$ of $Q_0$, for large enough $n$ the hyperbolic set $\Lambda_n$ satisfies the following property:
 every $z\in \Lambda_n$ has an iterate in $U_0$.
 \end{coro}

\section{Proof of \cref{ThmA}}\label{proof thm A}
To prove \cref{ThmA}, we are going to show the following stronger result:
\newtheorem*{thm:hom:tg}{\emph{Theorem \ref{thm hom tg}}}
\begin{thm:hom:tg}
For any  $k\geq 0$, there exists a family $(X_p)_{p\in\mathbb{I}^k}$ of Beltrami fields with $X_p\in\mathscr{B}(\R^3)$ that unfolds non-degenerately a homoclinic tangency of multiplicity $k$ at $p=0$.
\end{thm:hom:tg}

Indeed, we can easily show that \cref{thm hom tg} implies \cref{ThmA} using Duarte's Theorem:
\begin{proof}[Proof of \cref{ThmA}]

By \cref{thm hom tg}, there exists  a family $(X_p)_{p\in\mathbb{I}}$ of Beltrami fields with $X_p\in\mathscr{B}(\R^3)$ that unfolds non-degenerately a quadratic homoclinic tangency at $p=0$.
 Consider a Poincar\'e section for $(X_{p})_{p}$ around $p=0$ such that the associated family of return maps
$(f_{p})_{p\in \I}$ is a family of symplectic surface diffeomorphisms  unfolding non-degenerately a homoclinic tangency for a saddle  point $P$.
By   Duarte's theorem \cite[Theorem B]{Du} (here \cref{Duarte thm}),   at some $p_0$ small, the point $P$ belongs to a wild hyperbolic horseshoe which is non-degenerately unfolded.

Then every $\tilde X$ in a neighborhood $\cal V\subset  \mathscr B(\R^3)$ of $X_{p_0}$ induces a Poincar\' e map $\tilde f_{p_0}$ which displays a wild horseshoe.  Furthermore, if the neighborhood $\cal V$ is small enough, then the family $(\tilde X-X_{p_0}+X_p)_p$ is sufficiently close to $(X_p)_p$ to unfold non-degenerately  the wild hyperbolic horseshoe  at $p_0$.

Then by \cref{robust implies homo}, there is $p_1$ arbitrarily close to $p_0$ so that
$(\tilde X-X_{p_0}+X_p)_p$  unfolds non-degenerately a quadratic homoclinic tangency at $p=p_1$. Hence $\tilde X$  belongs to
the closure $H_\mathscr{B}(\R^3)$  of the set of Beltrami fields  that exhibit a saddle periodic orbit displaying a homoclinic tangency {and} such that the tangency can be non-degenerately unfolded by a family in $\mathscr B(\R^3)$. Thus $\cal V$ is included in $H_\mathscr{B}(\R^3)$ and so the interior $\cal N_{\mathscr B}(\R^3)$ of $H_\mathscr{B}(\R^3)$ is not empty.
\end{proof}
The proof of \cref{thm hom tg} follows exactly the strategy presented in \cref{sketch thm A}.

\subsection{Beltrami vector fields realizing prescribed jets of Melnikov functions at a strong, double heteroclinic link}\label{sec:LinALg}


We recall that in \cref{step 3 preuve}, we identified the $k$-jet of a real function $g$ at some point $\theta$, denoted by $J^k_\theta(g)$, to a vector in $\R^{k+1}$. Let $U\subset \R^3$ be a bounded open subset of $\R^3$. A consequence of \cref{thm:Meln:final}, which states that there exists an approximation of the right inverse of the  Melnikov operator in the space $\mathscr{B}(U)$ of Beltrami fields defined on $U$, is:

\begin{coro}\label{isoR}
Let $X\in\mathscr{B}(U)$ display a strong heteroclinic link $\Gamma\subset U$ and let $\Sigma\subset U$ be a Poincar\'e section endowed with coordinates, chosen as in \cref{thm:Meln:final}.
	For any $n\in\N$ and $N\in \N^*$, for every $N$-uplet of different points $\theta_1, \cdots , \theta_N\in \R/ \Z$,  the following map is onto:
	\[W\in \mathscr{B}(U) \mapsto (J^n_{\theta_i}(\mathcal{M}(W)))_{1\le i\le N} \in \R^{N\cdot(n+1)}\; .\]
\end{coro}

\begin{proof}
As the map $S: g\in C^\infty(\R/\Z,\R)\mapsto  (J^n_{\theta_i}(g))_{1\le i\le N} \in \R^{N\cdot(n+1)}$ is onto, there exists an $N\cdot(n+1)$-vector subspace $E\subset C^\infty(\R/\Z,\R)$ such that $S(E)=
\R^{N\cdot(n+1)}$. Now choose a basis  $(g_i)_{1\le i\le N\cdot(n+1)}$ of $E$. 	For every $1\leq i\leq N\cdot(n+1)$, by \cref{thm:Meln:final},
	 there exists a Beltrami field $W_{i}\in\mathscr{B}(U)$ such that $\cal M(W_{i})$ is $C^\infty$-close to $g_{i}$.
	 The image of the family  $(W_i)_{1\le i\le N\cdot (n+1)}$ by the  linear map
	\[
	  W\in     \mathscr{B}(U)\mapsto (J^n_{\theta_i}(\cal M(W)))_{1\leq i\leq N}\in 
	\R^{N\cdot(n+1)}
	\]
	is then close to the basis $(S(g_i))_{1\le i\le N\cdot(n+1)}$ and so forms a basis of $\R^{N\cdot(n+1)}$. This implies that the latter map is onto.
\end{proof}

%
%
%

Let now $X\in\mathscr{B}(U)$ display a double, strong heteroclinic link $(\Gamma^+,\Gamma^-)$ in $U\subset \R^3$. We choose two Poincar\'e sections $\Sigma_+\subset U$ and $\Sigma_-\subset U$ of the dynamics nearby $\Gamma^+$ and $\Gamma^-$,  endowed with coordinates,
 as it is considered in \cref{thm:Meln:final}.
 We fix reference points   $0\equiv p^+\in\Gamma^+$ and $  0\equiv p^-\in\Gamma^-$ to define
 the associated Melnikov operators  $\cal M_+$ and $\cal M_-$.

Let $k\geq 1$. In order to prove \cref{prop:transverse+k-het-tg}, we want to control the $1$-jet of the Melnikov function at some point of $\Gamma^+$ \emph{and} the $k$-jet of the Melnikov function at some point of $\Gamma^-$. We recall that it is not possible to simply apply \cref{isoR} separately on $\Gamma^+$ and $\Gamma^-$ and then sum the resulting Beltrami fields. Indeed, if we apply \cref{isoR} on $\Gamma^+$, we do not have any control on the behavior of the corresponding Beltrami vector field on $\Gamma^-$ and vice versa: each Beltrami field could possibly destroy the profile of the Melnikov function of the other one.
We will overcome this issue through a new trick based on elementary linear algebra.

\begin{proof}[Proof of \cref{prop:transverse+k-het-tg}] \label{proof of prop:transverse+k-het-tg}
	By \cref{isoR} applied at the heteroclinic link $\Gamma^+$ with $n=N=1$ and $\theta_1=\alpha\in\R/\Z$,  the following map  is onto:
	\[J_+:= W\in \mathscr B(U)\mapsto J^1_{\alpha}(\cal M_+(W))\in \R^2\; .\]
Again by \cref{isoR} applied at the heteroclinic link $\Gamma^-$ with $n=k$ and $N=3$  different points  $\theta_1,\theta_2,\theta_3\in\R/\Z$,
the following map is onto:
\[J_-:= W\in\mathscr B(U)\mapsto (J_1,J_2,J_3)(W) \in
\R^{3k+3}\text{ with } J_i(W):= J^{k}_{\theta_i} (\cal M_-(W))\in \R^{k+1}\; .\]
Equivalently,   the family  $(\ker J_1,\ker J_2,\ker J_3)$ is in general position:
\[\mathrm{codim} \bigcap_i  \ker J_i = \sum_i\mathrm{codim}  \ker J_i   \; .\]

To prove the \cref{prop:transverse+k-het-tg}, it suffices to show the existence of $1\le i\le 3$ such that the product map $J_-\times J_i$ is onto $\R^{k+3}$.
Otherwise, for every $i$, we have:
\[\coker \ker J_+ +\coker \ker J_i >  \mathrm{codim} ( \ker J_+ \cap  \ker J_i)\]
Thus   $\ker J_+ +  \ker J_i\neq \mathscr B(U)$. Let  $F_i$
be a codimension 1 space containing $\ker J_+ +  \ker J_i$.
The family of spaces  $(F_1,F_2,F_3)$  must be in general position since its elements include those of  the family  $(\ker J_1,\ker J_2,\ker J_3)$ which are in general position. Thus $\mathrm{codim} (F_1\cap F_2\cap F_3)=3 $.  This is a contradiction with  $F_1\cap F_2\cap F_3\supset \ker J_+$ which is 2-codimensional.\end{proof}

\subsection{Families of Beltrami fields unfolding non-degenerately a homoclinic tangency of multiplicity $k$}

Recall that $U\subset\R^3$ is a bounded open subset of $\R^3$. Let $X\in\mathscr{B}(U)$ display a strong heteroclinic link $\Gamma\subset U$. Let $\Sigma\subset U$ be a Poincar\'e section chosen as in \cref{thm:Meln:final}. The Melnikov operator $\cal M$ was introduced to study the \emph{displacement} operator, as explained in \cref{sec:Melnikov}. 
Choose  a reference point in $\Gamma\cap\Sigma$. Then, the displacement operator, denoted by $\mathrm{displ}$, is defined on a neighborhood $\cal U$ of $0\in\Gamma^\infty_\leb(U)$, and for every perturbation $W\in\cal U$ the function
\[
t\in [0,1]\mapsto \mathrm{displ}(W)(t)\in\R
\]
corresponds to the distance between the perturbed local stable and unstable manifolds over a fundamental domain, where this domain is parametrized by $t$.

Assume now that $U\subset \R^3$ is so that $\R^3\setminus U$ is connected. Let $X\in\mathscr{B}(U)$ display a double, strong heteroclinic link $(\Gamma^+,\Gamma^-)$ contained in $U$\footnote{This is the case of \cref{example:EPSR}, where $U$ is the intersection of $\R^3\setminus\{(0,0)\}\times\R$ with a large ball.}. Denote by $\mathrm{displ}_+$ and $\mathrm{displ}_-$ the displacement operators corresponding to the heteroclinic links $\Gamma^+$ and $\Gamma^-$ respectively.
From \cref{prop:transverse+k-het-tg} and since the Melnikov operator is the first order approximation of the displacement operator, we deduce the following

\begin{coro}\label{coro:het:tg}
Let $k\geq 0$. There exist $\alpha, \beta\in \R/\Z$, a $(k+3)$-vector subspace $\cal W\subset \mathscr{B}(U)$ and a neighborhood  $\cal U$ of $0$ in $\cal W$   such that the map
\[
W\in\cal U\mapsto \left( J^1_\alpha(\mathrm{displ}_+(W)),J^{k+1}_\beta(\mathrm{displ}_-(W)) \right)\in\R^{k+4}
\]
is a diffeomorphism onto a neighborhood of $0\in \R^{k+3}$. 
\end{coro}

\begin{proof}
Let $\alpha,\beta\in [0,1]$ and $\cal W\subset \mathscr{B}(U)$ be given by \cref{prop:transverse+k-het-tg}. The latter states that the following Melnikov operator is an isomorphism:
\[
	W\in \cal W\mapsto \left(J^1_{\alpha }(\cal M_+(W)), J^k_{\beta}(\cal M_-(W))\right)\in\R^{k+3}\, .
	\]
By the Poincar\'e-Melnikov Theorem (see \cref{sec:Melnikov}), the Melnikov operator is the differential of the displacement operator at $0$. Thus by the local inversion theorem, the displacement operator is a diffeomorphism from an open neighborhood $\cal U$  of $0$ onto an open neighborhood of $0\in\R^{k+3}$. 
\end{proof}

The proof of \cref{thm hom tg} follows then from \cref{coro:het:tg} and the para-inclination Lemma.

\begin{proof}[Proof of \cref{thm hom tg}]
First we extract from the family $(X+W)_{W\in \cal U}$ a $k$-parameter family
$(X+W)_{W\in \cal U'}$ such that $W^u_{loc} (\gamma^-,X+W) $  intersects transversally
$W^s_{loc} (\gamma^+,X+W) $ for every $W\in \cal U'$ and  $(W^s_{loc} (\gamma^-,X+W) )_{W\in \cal U'} $ unfolds non-degenerately a heteroclinic tangency of multiplicity $k$ with
 $(W^u_{loc} (\gamma^+,X+W) )_{W\in \cal U'} $.
 In order to do so, for a small   $\delta>0$, we consider the preimage, denoted by $\cal U'\subset \cal U$, of $\{0\}\times\{\delta\}\times (-\delta, \delta)^{k}\times \{0\} \times \{\delta\}$ by the diffeomorphism of \cref{coro:het:tg}. Note that for $\delta$ sufficiently small the map
\begin{equation}\label{restr:het}
W\in \cal U'\mapsto J^{k+1}_\beta(\mathrm{displ}_-(W))\in\R^{k}\times \{0\} \times \{\delta\}
\end{equation}
is still a diffeomorphism onto its image, which contains $\{0\}^k \times\times \{0\} \times \{\delta\}$. Thus there is indeed
a non-degenerate unfolding of  a heteroclinic tangency of multiplicity $k$  between $(W^s_{loc} (\gamma^-,X+W) )_{W\in \cal U'} $ and
 $(W^u_{loc} (\gamma^+,X+W) )_{W\in \cal U'} $.

Up to shrinking $\cal U'$, view in the Poincar\'e section $\Sigma$ defined above, we obtained a family of Poincar\'e maps $(f_W)_{W\in \cal U'}$ which displays two saddle fixed points $q^+$ and $q^-$ such that
$W^u_{loc} (q^-,f_W) $ intersects transversally $W^s_{loc} (q^+,f_W) $ for every $W\in \cal U'$, and so that  there is  a non-degenerate unfolding of  a heteroclinic tangency of multiplicity $k$  between $(W^s_{loc} (q^-,f_W) )_{W\in \cal U'} $ and
 $(W^u_{loc} (q^+,f_W) )_{W\in \cal U'} $.
Let $(S_W)_{W\in  \cal U'} $   be a family of segments contained in $(W^u_{loc} (q^-,f_W))_{W\in \cal U'} $ intersecting transversally  $(W^s_{loc} (q^+,f_W))_{W\in \cal U'} $. By the para-inclination   \cref{para-inclination lemma}, we can tune $(S_W)_{W\in  \cal U'} $ such that it is sent by a large  iterate $(f^n_W)_{W\in \cal U'}$ to a family $(f^n_W(S_W))_{W\in  \cal U'} $ which is $C^\infty$-close to $(W^u_{loc} (q^+,f_W))_{W\in \cal U'} $. As the  non-degeneracy of an unfolding is an open property, it comes that $(f^n_W(S_W))_{W\in  \cal U'} $ and $(W^s_{loc} (q^-,f_W))_{W\in \cal U'} $ unfold non-degenerately a tangency of multiplicity $k$.
As $S_W\subset W^u_{loc} (q^-,f_W) $ for every $W\in \cal U'$, we showed that $(f_{W})_{W\in \cal U'}$ unfolds non-degenerately a homoclinic tangency of multiplicity $k$ for $q^-$. This implies that there exists a $k$-parameter family of Beltrami vector fields $(X+W_p)_{p\in\I^k}$ in $\mathscr{B}(U)$, where each $W_p\in\cal U'$, which unfolds non-degenerately a homoclinic tangency of multiplicity $k$ for $\gamma^-$ at a parameter $p_0\in\I^k$. Denote by $Y:=X+W_{p_0}$, while $Z_i:=\partial_i W_{p_0}$ is the partial derivative along the direction of the $i$-th coordinate of $p\in\I^k$, for $i=1,\dots,k$. Since unfolding non-degenerately a homoclinic tangency of multiplicity $k$ is a robust condition, the family $(Y+p_1\cdot Z_1+\dots+p_k\cdot Z_k)_{(p_1,\dots,p_k)\in\I^k}$ unfolds non-degenerately a homoclinic tangency of multiplicity $k$ at some parameter in $\I^k$ close to $0$. Since $\R^3\setminus U$ is connected, \cref{GAT} provides the existence of vector fields $\tilde Y,\tilde Z_1,\dots,\tilde Z_k$ in $\mathscr{B}(\R^3)$ such that their restrictions to $U$ are close to $Y,Z_1,\dots,Z_k$ respectively. Again by the robustness of a non-degenerate unfolding of a homoclinic tangency of multiplicity $k$, the family of Beltrami fields $(\tilde Y+p_1\cdot \tilde Z_1+\dots+p_k\cdot\tilde Z_k)_{(p_1,\dots,p_k)\in\I^k}$ is the sought family in $\mathscr{B}(\R^3)$.

\end{proof}

\cn

\section{Proof of \cref{ThmD}}\label{proof thm D}
We are going to prove that  $\cal N_{\mathscr B}(\R^3)$ is GST wild.  In other words, we will prove that for every $J\ge 1$ and every non-empty open subset $\mathscr U$ of $\cal N_\mathscr{B}(\R^3)$, the following property holds true:
\begin{enumerate}[$(GST_{\mathscr U,J})$]
	\item
	there  is  a family $(X_p)_{p\in \I^J}$ of vector fields $X_p\in \mathscr U$ such that $X_0$ has a saddle periodic orbit $O$ which displays  $J$ different quadratic homoclinic tangencies unfolded  non-degenerately by the family.
\end{enumerate}

To this end we start with a family $(X_p)_{p\in \I}$ of vector fields in $\mathscr U$ which unfolds a quadratic homoclinic  tangency of a saddle periodic orbit $O$ at $p=0$. As the space of Beltrami fields is a vector space, we  can assume that $X_p=X_0+p\cdot X_1$ with $X_1:=\partial_p X_p|_{\{p=0\}}$.

We are going to use Duarte's \cref{Duarte thm} to create robust homoclinic tangencies for disjoint but homoclinically related hyperbolic basic sets.
Actually we would like moreover that their unions $C_p$ and $C_p'$ with their orbits of  homoclinic tangency are disjoint. Then the perturbative \cref{GATcoro} enables to create a new vector field $X_2$ which is small on $C_p$ but close to $X_1$  on  $C_p'$, and so   the family $(X_0+p_1\cdot X_1+p_2\cdot X_2)_{p=(p_1,p_2)\in \I^2}$  unfolds these two   basic sets  independently.  To obtain that $C_p$ and $C_p'$ are disjoint we will use     \cref{Duarte coro}  of Duarte's proof.

For the sake of simplicity, all along the proof, the hyperbolic continuation of an $f$-orbit $O$  for a perturbation $ \tilde f$ of $f$ will still be denoted by $O$.

\subsection{Case $J=2$}\label{creation disjoint wild horseshoe}
From the latter discussion,  to prove  $(GST_{\mathscr U,2})$, it suffices to show:
\begin{proposition} \label{main2casek1double}
	Let  $(X_0,X_1)$ be  Euclidean Beltrami fields and  let
	$X_p:= X_0+p\cdot X_1.$ Assume that    $(X_p)_{p\in \R}$  unfolds non-degenerately a quadratic   homoclinic tangency $\Gamma_0$ of a saddle orbit  $O$ at $p=0$. 
	Then there exist  an arbitrarily small  parameter   $p_0\in\I$, an Euclidean Beltrami field $X_{2}$ and a   parameter   $\tilde p_0\in\I^2$  arbitrarily close to  $(p_0,0)\in\I^2$  such that the family $(\tilde X_p)_{p\in \R^{2}}$ formed by $\tilde X_p:= X_0+p_1\cdot X_1+  p_{2} \cdot X_{2}$ unfolds non-degenerately
	$2$ different quadratic homoclinic tangencies of the saddle orbit $O$ at the parameter  $\tilde p_0$.
	%
	%
	%
\end{proposition}

\begin{proof}[Proof of \cref{main2casek1double}]
	For $p\in \I$, we recall that $X_p= X_0+p\cdot X_1$, where $X_0$ and $X_1$ are Beltrami fields.  Also  $(X_p)_{p\in \R}$ unfolds non-degenerately a quadratic   homoclinic tangency   $\Gamma_0$ \cn of a saddle  periodic orbit  $O$ at $p=0$.
	Let $\Sigma$ be a disk transverse to $X_p$ for every $p$ small, and intersecting  $O$ at a unique point $P$.  We assume that
	\[(cl(\Sigma)\setminus \Sigma)\cap (\Gamma_0 \cup O)=\emptyset\; .\]
	Then the return time to $\Sigma$ of the flow of  $X_0$ is smooth at some neighborhood $\check \Sigma\subset \Sigma$ of $\Sigma\cap (\Gamma_0 \cup O)$. For every $\epsilon>0$ small, up to taking $\check \Sigma$ slightly smaller,  the same occurs for the flows of $X_p$ for any $|p|<\epsilon$.

	Let $f_p: \check \Sigma\to \Sigma$ be the induced Poincar\'e map. Note that $(f_p)_{-\epsilon\le p\le \epsilon}$ is a $C^\infty$  family of maps in $\Diff^\infty(\check \Sigma, \Sigma)$.  
	It is standard that the map $f_p$ leaves invariant a symplectic form $\omega_p$  which depends smoothly on $p$. By Darboux's Theorem,  
	there are smooth coordinates $\phi_p:\D \to \Sigma$ such that $\phi_p$ sends $\leb$ into $\R\cdot \omega_p$. Then $F_p:= \phi_p^{-1}\circ f_p\circ \phi_p$ forms a smooth family of symplectic maps.  Moreover:
	\begin{fact}The  family $(F_p)_{-\epsilon\le p\le \epsilon}$ satisfies the assumptions of Duarte's \cref{Duarte thm}.\end{fact}
	Hence, for $\epsilon$ arbitrarily small, there exist   a non trivial  interval $I_0\subset (-\epsilon, \epsilon)$  and a continuation $(\Lambda_p)_{p\in I_0}$ of a hyperbolic horseshoe for $(f_p)_{p\in I_0}$   such that for every $p\in I_0$:
	\begin{itemize}
		\item $\Lambda_p$ is homoclinically related to the
		saddle fixed point $P$.
		\item $\Lambda_p$ displays a robust  homoclinic tangency.
	\end{itemize}
	Let $W^s_{loc}(\Lambda_p; f_p)= \bigcup_{z\in \Lambda_p} W^s_{loc}(z; f_p)$ and $W^u_{loc}(\Lambda_p; f_p)= \bigcup_{z\in \Lambda_p} W^u_{loc}(z; f_p)$ be the continuous families of local stable and unstable manifolds which are robustly tangent.
	Moreover by  \cref{Duarte coro}, we can assume that this unfolding  is non-degenerate.

	For any $p\in I_0$, there exists a  point   $z\in \Lambda_p$ such that
	$W^u_{loc}(z;f_p)$ is tangent to $W^s_{loc}(z;f_p)$.   Then there are segments of $W^u(P;f_p)$ and $W^s(P;f_p)$ which are arbitrarily close to
	$W^u_{loc}(z;f_p)$ and $W^s_{loc}(z;f_p)$. As the unfolding of the  tangency between $W^u_{loc}(z;f_p)$ and $W^s_{loc}(z;f_p)$ is non-degenerate,    there exists $p_0$ in the interior of $I_0$ such that $P$ displays an orbit of  homoclinic tangency $Q$ at $p=p_0$ which unfolds non-degenerately in  $(f_p)_{p\in I_0}$. Let  $Q_0$ be a point of $Q$ and let $U_0$ be a neighborhood of $Q_0$ which does not intersect $\Lambda_{p_0}$. 
	
	Now we apply a second time Duarte's \cref{Duarte thm} and moreover \cref{Duarte coro} for the homoclinic orbit of $Q$ and the point $Q_0\in Q$.  This gives the existence of  a non trivial segment $I_1\Subset I_0$ close to $p_0$ and a continuation of $(  \Lambda'_p)_{p\in I_0}$ of another hyperbolic horseshoe such that for every $p\in I_1$:
	\begin{enumerate}[$(i)$]
		\item The set $ \Lambda'_p$ is homoclinically related to $P$ and
so that any point of $\Lambda_p'$ has an iterate in $U_0$.
		\item The set $  \Lambda'_p$ displays a robust  homoclinic tangency which unfolds non-degenerately.
	\end{enumerate}
As $I_1$ is close to $p_0$, $\Lambda_p$ is close to $\Lambda_{p_0}$ and so does not intersect $U_0$. Hence none of the point $\Lambda_p$ has an iterate in $U_0$.  By $(i)$, this implies that $\Lambda_p$ and  $\Lambda'_p$ are disjoint for every $p\in I_1$.
	Let  $W^s_{loc}(\Lambda'_p; f_p)= \bigcup_{z\in\Lambda'_p} W^s_{loc}(z; f_p)$ \cn and $W^u_{loc}(\Lambda'_p; f_p)= \bigcup_{z\in \Lambda'_p} W^u_{loc}(z; f_p)$ be the continuous families of local stable and unstable manifolds which are robustly tangent and whose unfolding is non-degenerate. As $\Lambda'_p$ and $\Lambda_p\cup\{P\}$ are disjoint for every $p\in I_1$, we have:
	\begin{itemize}
		\item $W^s_{loc} ( \Lambda'_p; f_p)$ is disjoint from $W^s_{loc} ( \Lambda_p\cup\{P\}; f_p)$,
		\item $W^u_{loc} (  \Lambda'_p; f_p)$ is disjoint from $W^u_{loc} ( \Lambda_p\cup\{P\}; f_p)$.
	\end{itemize}
	We can assume  that all the local stable manifolds are compact; thus all the above sets are  compact. Thus for every $p\in I_1$:
	\begin{itemize}
		\item $K'^0_p:=W^s_{loc} ( \Lambda'_p; f_p)\cap W^u_{loc} (\Lambda'_p; f_p)$ and $  K_p^0:=W^s_{loc} (  \Lambda_p; f_p)\cap W^u_{loc} (  \Lambda_p; f_p)$ are disjoint.
		\item more generally $f^n_p(K'^0_p)$ and $f^m_p(K_p^0)$ are disjoint for every $n, m\in \Z$.
	\end{itemize}
	As $f^n_p(K'^0_p)\to \Lambda_p'$ and $f^n_p(K_p^0)\to \Lambda_p$ when $n\to \pm \infty$ we obtain:
	\begin{fact} \label{disjoint Kp} For any $p\in I_1$,
		$K_p:= \bigcup_{n\in \Z} f^n_p (K_p^0) $ and  $K'_p:= \bigcup_{n\in \Z} f^n_p (K'^0_p) $ are disjoint  compact~sets.\end{fact}
	Let us precise the topology of these sets:
	\begin{fact} \label{disconected}
		The subsets $K'_p$ and $K_p$ are totally disconnected.
	\end{fact}
	\begin{proof} Let us prove that $K_p$ is   totally   disconnected; the proof for $K'_p$ is identical.
		First,  $\Lambda_p$  is totally disconnected since it is a horseshoe and so a Cantor set.
		
		Let us show that  $ K_p^0 \setminus \Lambda_p$ is totally disconnected. Indeed,
		$ K_p^0 \setminus \Lambda_p$ is formed by the quadratic tangency points between
		$ W^s_{loc} (  \Lambda_p; f_p)$ and $ W^u_{loc} (  \Lambda_p; f_p)$.  These are a totally disconnected union of leaves  ($\Lambda_p$ is a horseshoe). Hence the set of quadratic tangencies between the leaves of  $ W^s_{loc} (  \Lambda_p; f_p)$ and $ W^u_{loc} (  \Lambda_p; f_p)$ is totally disconnected.
Thus  $K_p:= \Lambda_p\cup \bigcup_{n\in\Z} f_p^n(K^0_p\setminus  \Lambda_p)$ is a compact set equal to a countable union of  totally disconnected compact subsets. By
\cite[II.4.A]{HurWal41} any compact  set is
totally disconnected  iff it is zero dimensional in the sense of \cite[Def. II.1]{HurWal41}: any of its point has an arbitrarily small  clopen neighborhood. By \cite[Thm II.2.]{HurWal41},   a countable union of closed zero-dimensional sets is still zero-dimensional; thus the compact set $K_p$ is zero-dimensional and so   totally disconnected \cite{HurWal41}.		
	\end{proof}
	We have furthermore:
	\begin{lemma} \label{eta distant check K K}
		The maps $p\in I_1\mapsto  K_p$ and $p\in I_1\mapsto K'_p$ are upper semi-continuous: for every $p_1\in I_1$, for every neighborhood $U$ of $K_{p_1}$  (resp. $K'_{p_1}$), the compact set $K_p$ (resp.   $K'_p$) is included in $U$ for $p$ close to $p_1$.
	\end{lemma}
	\begin{proof}We prove the statement for  $p\in I_1\mapsto  K_p$; the proof for  $p\in I_1\mapsto K'_p$ is the same.
		It suffices to notice that $\bigcup_{p\in I_1} \{p \}\times K_p$ is compact. This is indeed the case since the local stable and unstable sets of $\Lambda_p$ vary continuously with $p$. So
		$\bigcup_{p\in I_1} \{p \}\times W^s_{loc} (K_p;f_p)$  and  $\bigcup_{p\in I_1} \{p \}\times W^u_{loc} (K_p;f_p)$      are compact, and their intersection
		$\bigcup_{p\in I_1} \{p \}\times K^0_p$   is compact.
		Hence $p\mapsto K_p^0$   is upper continuous. Thus  $p\mapsto f^n(K_p^0)$ is upper semi continuous and $p\mapsto \Lambda_p\cup \bigcup_{-N\le n\le N} f^n(K_p^0)$ is upper semi continuous for every $N$. As this sequence of functions converges uniformly to $p\mapsto K_p$ when $|n|\to \infty$, the latter is upper semi-continuous.  \end{proof}
	
	We recall that $K_p$ and $K'_p$ are included in $\check \Sigma$ for every $p\in I_1$. Also the return time $\tau_p$ from $\check \Sigma$ to $\Sigma$ is smooth. Let $(\phi_p^t)_{t\in \R}$ be the flow of $X_p$. We define for $p\in I_1$:
	\[C_p:= \bigcup_{z\in K_p} \{\phi_p^t(z): 0\le t\le \tau_p(z)\}\qand
	C'_p:= \bigcup_{z\in K'_p} \{\phi_p^t(z): 0\le t\le \tau_p(z)\}\; .\]
	By \cref{disjoint Kp}  it holds:
	\begin{fact}\label{compact disjoint} For every $p\in I_1$ the compact subsets $C_p$ and $C'_p$ of $\R^3$ are disjoint.\end{fact}
	By \cref{eta distant check K K} and the smoothness of $\tau$, we have:
	\begin{fact}\label{semi continuous} The maps $p\in I_1\mapsto  C_p$ and $p\in I_1\mapsto C'_p$ are upper semi-continuous.
	\end{fact}
	
	Now we prove:
	{
		\begin{lemma}\label{lemma conn compl}
			Let $p_0\in I_1$. The compact subsets $C_{p_0}$ and $C'_{p_0}$ of $\R^3$ have connected complement in $\R^3$.
		\end{lemma}
		\begin{proof}
			By symmetry, we show the lemma only for $C_{p_0}$. Assume by contradiction that $\R^3\setminus C_{p_0}$ has a bounded component $F$. As the flow of $X_{p_0}$ leaves invariant $C_{p_0}$, it leaves also the component $F$ of $\R^3\setminus C_{p_0}$ invariant. Thus the component $F$ intersects $\Sigma\setminus K_{p_0}$.
Recall that $C_{p_0}$ is included in the suspension of $\check \Sigma\Subset \Sigma$, which is isotopic to an embedded solid torus and so its complement is connected. Hence  $F\cap \Sigma$ cannot be connected to the boundary of $\Sigma$ by a path in $\Sigma\setminus K_{p_0}$.  This implies that $\Sigma\setminus K_{p_0}$ is not connected. This is a contradiction with the connectedness \cite[Theorem 4, Page 93]{Moise77} of the complement in a disk $\Sigma$ of any totally disconnected compact set, such as $K_{p_0}$, by \cref{disconected}.
		\end{proof}
	}
	
	%
	
	\begin{figure}[h]
		\centering
		\includegraphics[scale=0.5]{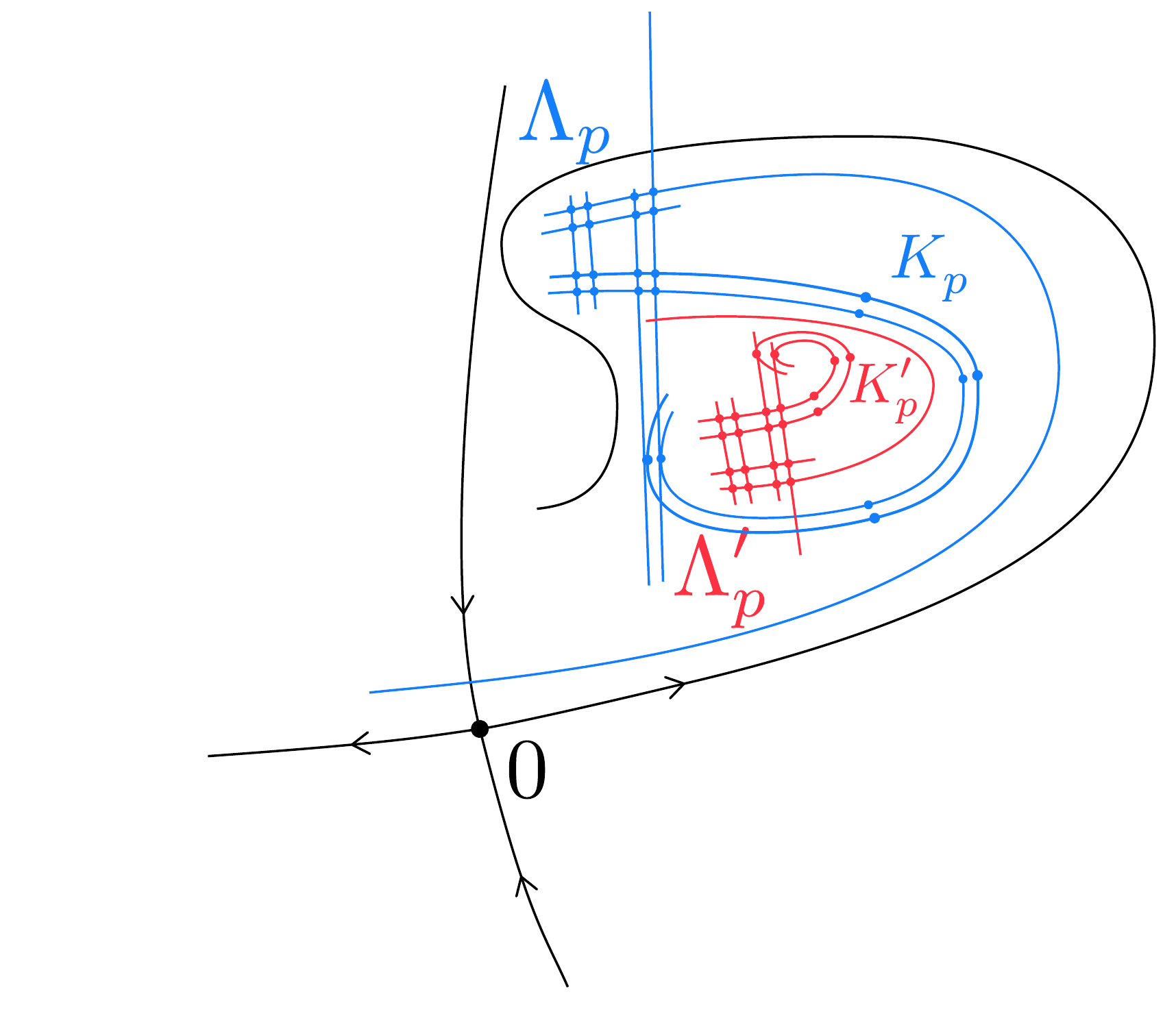}
		\caption{Basic hyperbolic sets $\Lambda_p$ and $\Lambda_p'$   displaying disjoint  homoclinic sets $K_p$ and $K_p'$.}
		\label{disjoint neigh}
	\end{figure}

	{
		Now we can apply the Global Approximation Theorem to construct $X_2$. More precisely, we apply \cref{GATcoro} with $C_{p_0}$ and $C'_{p_0}$ as respectively $K_+$ and $K_-$. They are disjoint compact subsets of $\R^3$ by \cref{compact disjoint} and they both have connected complement in $\R^3$ by \cref{lemma conn compl}. Thus \cref{GATcoro} provides the existence of a Beltrami field $X_2$ whose restriction to $C_{p_0}\sqcup C'_{p_0}$ is arbitrarly close to
		\[{ \tilde X_2:z\in V\sqcup V'\mapsto}
		\begin{cases}
		0 &\text{ if }z\in V, \\
		X_1(z) &\text{ if }z\in V'\, ,
		\end{cases}
		\]
		where $V$ and $V'$ are disjoint neighborhoods of respectively $C_{p_0}$ and $C'_{p_0}$.}
	%
	%
	We recall that for a dense subset of parameters $p_0\in I_1$, the vector  field $X_0+p_0 \cdot X_1$ displays a quadratic homoclinic tangency for a saddle periodic orbit $O_1\in \Lambda_{p_0}$ which unfolds non-degenerately. Also for $ p_0'$ arbitrarily small, the vector field $X_0+(p_0+  p_0')\cdot  X_1$ unfolds  non-degenerately a quadratic homoclinic tangency for a  saddle periodic orbit  $O_2\in \Lambda'_{p_0+ p_0'}$. Note that
	$X_0+p_0 \cdot  X_1+ p_0' \cdot  \tilde X_2$ displays
	two quadratic homoclinic tangencies with the orbits $O_1$ and $O_2$ and which are non-degenerately unfolded in the family $(X_0+p_1\cdot X_1+ p_2 \cdot  \tilde X_2)_{(p_1,p_2)\in\I^2}$.
	Indeed by \cref{semi continuous}, the closure of the orbit of the tangencies $C_p'$  of $\Lambda'_p$ is in $V'$  for $p$  close to $p_1$.
	
	As this is a robust condition, for $X_2$ close enough to $\tilde X_2$, there exists  $\tilde p_0'$ nearby $(p_0,p_0')$  (and so nearby $(p_0,0)$) such that the family $(\tilde X_{p})_{p=(p_1,p_2)\in \I^{2}}$ formed by $\tilde X_p:= X_0+p_1\cdot X_1+ p_{2} \cdot X_{2}$ displays at  $p=\tilde p_0'$   two quadratic homoclinic tangencies which unfold  non-degenerately for, respectively, the hyperbolic  continuations of $O_1$ and $O_2$.

	It remains to prove the existence of two quadratic homoclinic tangencies which unfold  non-degenerately for the same periodic saddle $O$. To this end,  we recall that $O_1\in C_{\tilde p_0}$ and $O_2\in C'_{\tilde p_0}$  are homoclinically related to $O$.
	This implies that for every $1\le k\le 2$, there are segments $S_k$ and $U_k$ of $W^s(O ; f_{\tilde p_0'})$ and $W^u(O ; f_{\tilde p_0'})$  which are close to $W^s_{loc}(O_k; f_{\tilde p_0'})$ and $W^u_{loc}(O_k; f_{\tilde p_0'})$.
	Moreover, by the para-inclination lemma  \cite{Berger2016}, when $p$ varies, the hyperbolic continuation of  these segments forms families which are $C^2$-close to the hyperbolic continuation of  $W^s_{loc}(O_k; f_{\tilde p_0'})$ and $W^u_{loc}(O_k; f_{\tilde p_0'})$. In particular the relative positions are $C^1$-close. Thus by the local inversion Theorem there is $\tilde p_0$ arbitrarily close to $\tilde p_0'$ such that  the continuation of $S_k$ has a quadratic tangency with $U_k$ for every $1\le k\le 2$, and these quadratic homoclinic tangencies are unfolded non-degenerately when $p$ varies.
\end{proof}

\subsection{Case $J\ge 3$}\label{sec:main2induc}

To obtain  $(GST_{\mathscr U,J})$ for every $J\ge 1$, we will apply inductively on $k$ the following  generalization of \cref{main2casek1double}:

\begin{proposition} \label{premain2}
Let $k\ge 1$. 	Let $(X_0,X_1,\dots, X_k)$ be  Euclidean Beltrami fields and   for every $p=(p_1,\dots, p_k)\in \R^k$, let $X_p:= X_0+p_1\cdot X_1+\cdots +p_k\cdot X_k$.
	Assume that there exists a saddle periodic  orbit $O$   for $X_0$ which  displays  $k$ quadratic homoclinic tangencies which unfold non-degenerately in the family $(X_p)_{p\in \R^k}$.
	
	Then there exist   an arbitrarily small  parameter $p_0\in \R^k$, an Euclidean Beltrami field $X_{k+1}$ and a   parameter $\tilde p_0\in \R^{k+1}$ arbitrarily close to $(p_0,0)\in \R^{k+1}$  such that the family $(\tilde X_p)_{p\in \I^{k+1}}$ formed by $\tilde X_p:= X_0+p_1\cdot X_1+\cdots +  p_{k+1} \cdot X_{k+1}$ unfolds non-degenerately
	 $k+1$ different quadratic homoclinic tangencies to $O$ at the parameter  $\tilde p_0$.

	\end{proposition}
\begin{proof}[Proof of \cref{premain2}]We proved the case $k=1$ in \cref{main2casek1double}. 	Now assume $k\ge 2$.

	By non-degeneracy of the unfolding,  there is a path $t\in \I \mapsto p(t)\in \I^k$  with $p(0)=0$ so that the family $(X_{p(t)})_{t\in \I}$  unfolds non-degenerately
	the $k^{th}$ homoclinic tangency $\Gamma_k$   but leaves  the other tangencies $\Gamma_i$ persistent for $1\le  i<k$.

Then we proceed similarly to the case $k=1$. For $t'_0$ arbitrarily small, at  $p_0'=p(t'_0)$, Duarte's theorem provides a hyperbolic set $\Lambda$ which is homoclinically related to $O$ and displays a robust quadratic homoclinic tangency which unfolds non-degenerately. Up to perturbing $p_0'$, we can assume that one of these new homoclinic tangencies is given by a periodic orbit $O'\neq O$.
Hence we obtain for $t'_0$ arbitrarily small,
the existence of a new orbit of quadratic  homoclinic tangency $\Gamma'$  to a periodic saddle $O'\neq O$ which unfolds non-degenerately at $p'_0=p(t'_0)$, and so that $O'$ is homoclinically related to $O$.

Observe that $\Gamma'\cup O'$ is a compact  set which is disjoint from $\bigcup_{1\le i< k} O \cup \Gamma_i$.
Observe that both $\Gamma'\cup O'$ and $\bigcup_{1\le i< k} O \cup \Gamma_i$ have connected complement in $\R^3$.
Let $U'$ and $U$ be small disjoint neighborhoods of respectively $\Gamma'\cup O'$ and $\bigcup_{1\le i < k}O\cup \Gamma_i$.
	
	We apply \cref{main2casek1double} to the one parameter family   $(X_{p(t)})_{t\in \I}$. It provides $p_0''=p(t''_0)\in \R^k$ close to $p_0'=p(t_0')$ (and so   small), 	a  vector field $X_{k+1}$, and $(t_0''', s_0)  \in \R^2$ arbitrarily close to $(t''_0, 0)$, such that
at 	$(t,s)= (t_0''', s_0)$, the family
	 $  (X_{p(t)}+s\cdot X_{k+1})_{(t,s)\in \R^2}$  unfolds non-degenerately two different orbits of homoclinic tangency at $O'$ that are contained in $U'$.
	
At $(p,s)=(p(t_0'''), s_0)$, the family $(X_{p}|_ {U\sqcup U'}+s\cdot  1\!\!1_{U'}\cdot X_{k+1})_{(p,s)\in \R^k\times \R} $  displays
$(k-1)$-orbits of quadratic tangency to $O$ and 2 orbits of homoclinic tangency to $O'$. They unfold non-degenerately when $(p,s)$ varies in $\R^{k+1}$.

%

Using then \cref{GATcoro} with $\Gamma'\cup O'$ and $\bigcup_{1\le i < k}O\cup\Gamma_i$ as respectively $K_+$ and $K_-$, and with $1\!\!1_{U'}\cdot X_{k+1}$ as Beltrami field on the neighborhood $U\sqcup U'$ of $K_+\sqcup K_-$, we deduce the existence of a Beltrami field $\tilde{X}_{k+1}\in\mathscr{B}(\R^3)$ arbitrarily close to $1\!\!1_{U'}\cdot X_{k+1}$ on $ U\sqcup U'$.
As having $(k+1)$-homoclinic tangencies which unfold non-degenerately is an open condition on families, this property persists. Up to rename $\tilde X_{k+1}$ as $X_{k+1}$, that is why we can assume $X_{k+1}$ defined on $\R^3$.

Finally we use again the (para)-inclination lemma 
 and the fact that $O$ and $O'$ are homoclinically related to deduce from the  non-degenerate unfolding of the two quadratic homoclinic tangencies to  $O'$,  that two new quadratic homoclinic tangencies to $O$   are obtained   together with the $(k-1)$-previous quadratic homoclinic tangencies to $O$, and they  unfold non-degenerately.
\end{proof}

\begin{appendix}
	
\section*{Appendices}

\section{Proof of \cref{prop:Meln:smooth}: right  inverse of the  Melnikov operator}\label{sec:Melnikov}

The aim of this appendix is to construct an inverse of the Melnikov operator for a vector field $X\in\Gamma^\infty_\leb(U)$ displaying a strong heteroclinic link in $U$, for some open subset $U$ of $\R^3$.\cn

\subsection{Displacement and Melnikov operators}\label{section Melnikov}
The aim of this section is to introduce the definitions of displacement operator, Melnikov operator and explain their relation.

Let $\Sigma\subset \R^2$ and let $\Omega$ be the standard symplectic, smooth 2-form on $\Sigma\subset\R^2$.
Let $f\in C^\infty(\Sigma,\R^2)$ be a smooth, symplectic diffeomorphisms from $\Sigma$ onto its image. We assume that $f$ displays two hyperbolic fixed point $p^+$, $p^-$ and a heteroclinic link $L$ from $p^-$ to $p^+$.
The \emph{displacement operator} will describe the unfolding of $L$ when we consider a perturbation of $f$.

Let $\cal N\subset C^\infty(\Sigma,\R^2)$ be a neighborhood of $0$. Assume that $\cal N$ is small enough such that, for every $\epsilon\in \cal N$, the diffeomorphism $f+\epsilon$ admits unique hyperbolic fixed points $p^\pm_\epsilon$ nearby $p^\pm$, called their hyperbolic continuations.
Let $\ell:\R\to L$ be a smooth parametrization of the heteroclinic connection $L$ such that
\[
\lim_{t\to\pm \infty}\ell(t)=p^\pm\quad\text{and}\quad f(\ell(t))=\ell(t+1).
\]
Fix local unstable and local stable manifolds $W^u_{loc}(p^-;f)$ and $W^s_{loc}(p^+;f)$, of respectively $p^-$ and $p^+$, which both contain $\ell([0,1])$.

For a fixed $\epsilon\in C^\infty(\Sigma,\R^2)$ small, consider the smooth immersed submanifold $W^s_{loc}(p^+_\epsilon; f+\epsilon)$. We can find a smooth diffeomorphism $\phi^s_\epsilon:W^s_{loc}(p^+; f)\to W^s_{loc}(p^+_\epsilon; f+\epsilon)$.
Similarly, we find a smooth diffeomorphism $\phi^u_\epsilon:W^u_{loc}(p^-; f)\to W^u_{loc}(p^-_\epsilon; f+\epsilon)$. By Irwin's proof \cite[Theorem 28]{Ir72} of the stable/unstable manifold theorem (which is based on the implicit function theorem), both $\phi^s_\epsilon$ and $\phi^u_\epsilon$ can be chosen so that they depend smoothly on $\epsilon$.
Up to a smooth reparametrization, we can assume that
\begin{equation}\label{commute:dynamics}
(f+\epsilon)\phi^u_\epsilon(\ell(0))=\phi^u_\epsilon(\ell(1))\qquad\text{and}\qquad (f+\epsilon)\phi^s_\epsilon(\ell(0))=\phi^s_\epsilon(\ell(1))\, .
\end{equation}
We define, for $t\in[0,1]$,\footnote{Equivalently, we can express the displacement function as the vector product $\partial_t\ell(t)\wedge (\phi^u_\epsilon(\ell(t))-\phi^s_\epsilon(\ell(t)))$.}
\begin{equation}\label{displacemtn function}
\mathrm{displ}(\epsilon)(t)=\Omega(\partial_t\ell(t),\phi^u_\epsilon(\ell(t))-\phi^s_\epsilon(\ell(t))).
\end{equation}

Roughly speaking, $\mathrm{displ}(\epsilon)(t)$ is the algebraic distance, along an affine line passing through $\ell(t)$ and normal to $L$, between the perturbed local stable and unstable manifolds. See \cref{figDispl}.

\begin{figure}[h]
	\centering
	\includegraphics[scale=0.7]{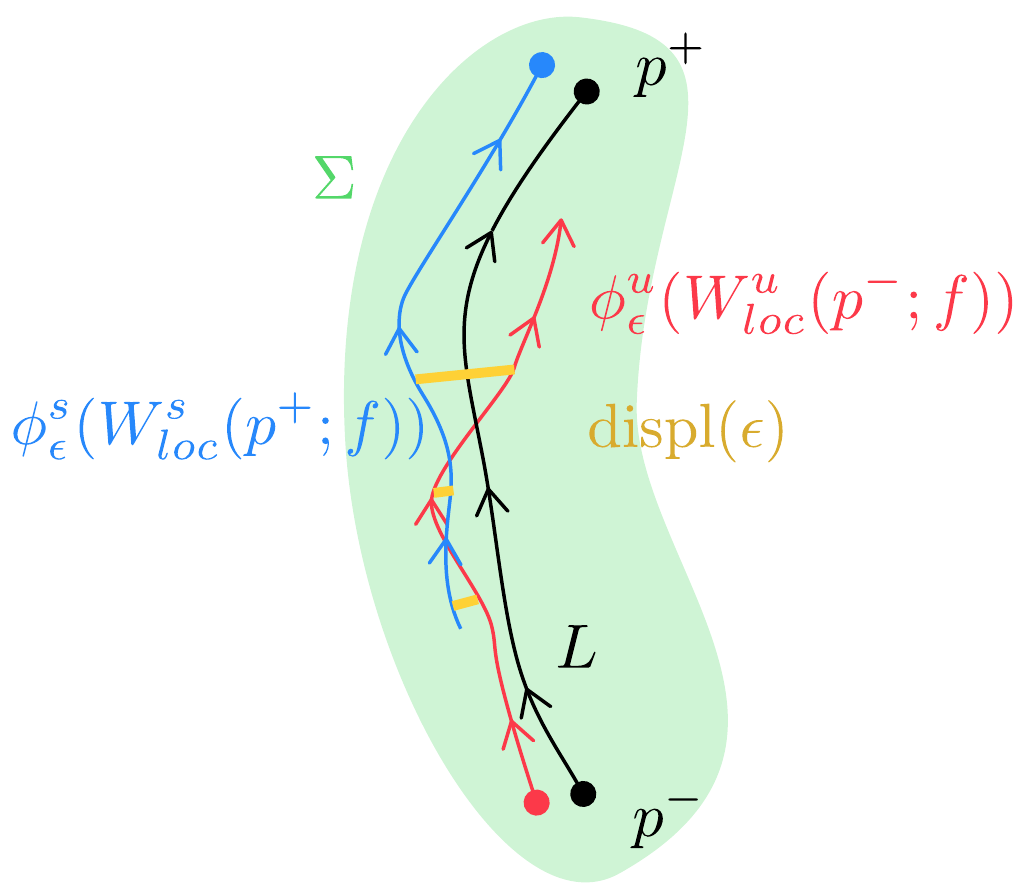}
	\caption{The displacement function is the algebraic distance between perturbed local stable and unstable manifolds.}
	\label{figDispl}
\end{figure}

By \cref{commute:dynamics}, $\phi^u_\epsilon(\ell([0,1]))$ and $\phi^s_\epsilon(\ell([0,1]))$ are fundamental domains. Then, the function $\mathrm{displ}(\epsilon)$ vanishes if and only if $\phi^s_\epsilon(\ell(t))= \phi^u_\epsilon(\ell(t))$ for every $t\in[0,1]$.  Thus, the heteroclinic link persists if and only if $\mathrm{displ}(\epsilon)\equiv 0$.

\begin{definition}
	The \emph{displacement operator} is the following map from a neighborhood $\cal N$ of $0$ in $C^\infty(\Sigma,\R^2)$:
	\[
	\mathrm{displ}:\cal N\to C^\infty([0,1],\R)
	\]
	\[
	\epsilon\mapsto \mathrm{displ}(\epsilon),
	\]
	where the function $\mathrm{displ}(\epsilon)$ is defined in \cref{displacemtn function}.
\end{definition}

\begin{remark}
	The displacement operator is not uniquely defined: it depends on the parametrization $\ell$ of $L$ and on the families $(\phi^u_\epsilon)_\epsilon$ and $(\phi^s_\epsilon)_\epsilon$.
\end{remark}

We recall that $\phi^s_\epsilon$ and $\phi^u_\epsilon$ depend smoothly on $\epsilon$, so the following proposition holds.

\begin{proposition}
	The map $\mathrm{displ}:\cal N\to C^\infty([0,1],\R)$ is smooth.
\end{proposition}

The distance between the perturbed stable and unstable manifolds can be studied through the so-called Melnikov functions. We present here the main definitions and properties and refer for detailed proofs to \cite{LMRR}, \cite{DRR}, \cite{DRR2}, \cite{GPB}; the historical references are \cite{Poi} and \cite{Mel}.

\begin{definition}\label{def Melnikov}
	Let $f\in C^\infty(\Sigma,\R^2)$ be a symplectic diffeomorphism, which displays a heteroclinic link $L$, parametrized by $\ell:\R\to L$. The \emph{Melnikov function} associated to a smooth family $r\mapsto \epsilon(r)\in C^\infty(\Sigma,\R^2)$, with $\epsilon(0)=0$, is: 
	\begin{equation}\label{integral Melnikov}
	\begin{aligned}\mathcal{M}(\epsilon(s)):&\ \R\to \R \\
	t\mapsto \cal{M}(\epsilon(s))(t ):=&\,\sum_{k\in\Z}\Omega\left( \partial_t\ell(t), (f^*)^k(\partial_r\epsilon(r)|_{r=0}\circ f^{-1})\circ \ell(t)\right)\, ,\end{aligned}
	\end{equation}
	where $\partial_r\epsilon(r)|_{r=0}:=\lim_{r\to 0}\frac{\epsilon(r)}{r}$.
\end{definition}

\begin{remark}
	In \cref{integral Melnikov} the presence of $\partial_r\epsilon(r)|_{r=0}\circ f^{-1}$ is due to the fact that, for the given family $(f+\epsilon(r))_r$, it is the vector field $X=\partial_r\epsilon(r)|_{r=0}\circ f^{-1}$ which satisfies $\frac{d}{dr}(f+\epsilon(r))|_{r=0}=X\circ f$. See \cite[Section 2]{LMRR}.
\end{remark}

Recall that $f$ is symplectic and, since $f(\ell(t))=\ell(t+1)$, $\partial_t\ell(t)= Df^{-k}(\ell(t+k))\partial_t\ell (t+k)$. Thus, it holds
\begin{equation}\label{formula melnikov}
\cal{M}(\epsilon(s))(t)=\sum_{k\in\Z}\Omega\left( \partial_t\ell(t+k), \partial_s\epsilon|_{s=0}(\ell(t+k-1))\right).
\end{equation}

Note that the Melnikov function belongs to the space $C^\infty(\R/\Z,\R)$ of smooth $1$-periodic functions.

\begin{definition}\label{def Melnikov operator}
	Let $f$ be a symplectic diffeomorphism in $C^\infty(\Sigma,\R^2)$ displaying a heteroclinic link $L$, parametrized by $\ell:\R\to L$. The \emph{Melnikov operator} is a linear and continuous map 
	\[
	\cal{M}:C^\infty(\Sigma,\R^2)\ni \epsilon\mapsto \cal M(\epsilon)\in C^\infty(\R/\Z,\R)\, ,
	\]
	where $\cal M(\epsilon)$ is the Melnikov function with respect to the smooth family $r\mapsto r\epsilon$.
\end{definition}
\begin{remark}\label{remark just restriction}
	Observe that $\cal M(\epsilon)$ depends only on $\epsilon|_L$. Thus, we can define $\cal M(\epsilon)$ for $\epsilon\in C^\infty(L,\R^2)$.
\end{remark}

\begin{theorem}[Poincar\'e-Melnikov, \cite{Poi}-\cite{Mel}]\label{Melnikov}
	Let $f$ be a symplectic diffeomorphism in $C^\infty(\Sigma,\R^2)$, displaying a heteroclinic link $L$, parametrized by $\ell:\R\to L$. Then, the partial differential of the displacement operator at $0$ is equal to the Melnikov operator at $\epsilon$, i.e.
	\begin{equation}
	D_0\mathrm{displ}=\cal{M}.
	\end{equation}
	Thus, for any $r\geq 1$, in the $C^r$-uniform norm it holds
	\[
	\mathrm{displ}(\epsilon)=	\mathrm{displ}(0)+	D_0\mathrm{displ}(\epsilon)+O(\epsilon^2)=\cn\cal{M}(\epsilon)+O(\epsilon^2).
	\]
	
\end{theorem}
For a detailed proof, we refer for example to \cite{DRR} and \cite{GPB}.

\subsection{Inverse Melnikov operator for flows in $\R^3$}\label{subsec:Meln:inverse}
Let $U$ be an open subset of $\R^3$. Let $X\in \Gamma^\infty_\leb(U)$ display a \emph{strong heteroclinic link} $\Gamma\subset U$ between two saddle periodic orbits $\gamma^\pm\subset U$. In particular, the closure of the link $\bar\Gamma$ is diffeomorphic to $\R/\Z\times[-1,1]$.

Let $\bar L$ be the closure of the half strong stable manifold of a point in $\gamma^+$ in $\Gamma$. Let $\hat \Sigma\subset \R^3$ be a disk which contains $\bar L$ and that intersects transversally $\Gamma$.
Up to reducing $\hat \Sigma$, it is   transverse to the vector field $X$. Let $\Sigma\Subset \hat \Sigma$ be a neighborhood of $L$ in $\hat\Sigma$ such that for every point $z \in \Sigma$, there exists $\tau(z)>0$ minimal such that the orbit of $z$ intersect transersally $\hat \Sigma$ at time $\tau(z)$ at a point $f(z)$. From the definition of strong heteroclinic link, the first return time $\tau|_{ \Sigma\cap\Gamma}$ is constant, in particular $\gamma^+$ and $\gamma^-$ have the same period. Indeed, since $\Sigma\cap\Gamma=L$ is a half strong stable manifold of a point in $\gamma^+$, it is $T(\gamma^+)$-periodic, where $T(\gamma^+)$ denotes the period of $\gamma^+$ (similarly, for a half strong unstable manifold of a point if $\gamma^-$).
The disks $\Sigma\subset \hat \Sigma$ are called Poincar\'e sections and $f$ is  called a Poincar\'e first return map.  The maps $f$ and $\tau$ are in, respectively, $C^\infty(\Sigma,\hat \Sigma)$ and $ C^\infty(\Sigma,\R^+)$. Moreover, it is standard that there exists a smooth symplectic form $\Omega$ in $\hat\Sigma$ which is invariant by the action of $f|_\Sigma$. By Darboux's Theorem, we can identify $\Sigma$ with a disk of $\R^2$ so that $\Omega$ is the standard symplectic form.

Let $W\in\Gamma^\infty_\leb(U)$ be a small perturbation of $X$: it induces a small perturbation $\epsilon \in C^\infty(\Sigma,\R^2)$ of the Poincar\'e map $f$. So we can study the unfolding of $\Gamma$ using the tools of Melnikov operator for the heteroclinic link $L=\hat \Sigma\cap\Gamma$ presented above.

Fix a parametrization $\ell:\R\to L$ of $L$, such that $\lim_{t\to\pm\infty}\ell(t)=p^\pm$, where $p^\pm:=\hat \Sigma\cap \gamma^\pm$, and $f(\ell(t))=\ell(t+1)$. Fix also $(\phi^s_\epsilon)_\epsilon$ and $(\phi^u_\epsilon)_\epsilon$ as in \ref{section Melnikov}. Let $\cal M$ be the Melnikov operator associated to this setting.
\begin{definition} The \emph{Melnikov operator} is the following linear and continuous map from a neighborhood $\cal N$ of $0$ in $\Gamma^\infty_\leb(U)$
	\[ \cal{M}: \cal N\ni W\mapsto
	\cal M(W):=\cal{M}(\epsilon)\in C^\infty(\R/\Z,\R)\, ,\]
	where $\epsilon\in C^\infty(\Sigma,\R^2)$ is the perturbation of the first return map $f$ associated to the vector field $W$.
\end{definition}

\begin{remark}
	Observe that, since $\cal M(\epsilon)$ depends only on $\epsilon|_L$ (see \cref{remark just restriction}), and so on $W|_\Gamma$, the operator $\cal M$ can be defined for $W\in C^\infty(\Gamma,\R^3)$.
\end{remark}

The aim of this subsection is to give the proof of \cref{prop:Meln:smooth}: such a proposition gives an inverse operator of the Melnikov one.

\begin{proof}[Proof of \cref{prop:Meln:smooth}]
		We recall that $\Gamma$ is an invariant cylinder smoothly embedded in $\R^3$, and that
		$L\subset \Gamma$ is a line transverse to the flow whose return time is a constant $T>0$.
As the statement is invariant by reparametrizing the flow, we can assume that $T=1$.
		
This defines a canonical  diffeomorphism	$\Gamma\to L\times \R/\Z$ which maps $L$ to $L\times \{0\}$. Recall that $\ell:\R\to L$ is a parametrization of $L$ such that $f\circ\ell(t)=\ell(t+1)$  for every $t\in\R$, where $f$ is the first return map to $  \Sigma$. Then using $\ell$, we identify  $\Gamma$ to $\R\times\R/\Z$ and $L$  to $\R\times \{0\}$:
\[ \Gamma\equiv \R\times\R/\Z \qand L\equiv \R\times \{0\}
\; ,\]
such that, with  $p_2: \Gamma \to \R/\Z$ the  second coordinate projection, the flow $(\Phi^t_X)_{t\in\R}$ of   $X$ satisfies:
	\[
	p_2\circ \Phi^t_X(s,\theta)= \theta+t\ \mod 1\; , \quad \forall (s,\theta)\in  \R\times\R/\Z\equiv \Gamma\, .
	\]
	
	\begin{lemma}\label{prop Psi}
		There is a compactly supported function $\Psi\in C^\infty(\R\times\R/\Z,\R^+)$ such that:
		\begin{equation*}
		\int_\R\Psi\circ\Phi^t_X(s,0)dt=1\;, \quad \forall s\in \R.
		\end{equation*}
	\end{lemma}
	\begin{proof}
		Let $\Psi_0\in C^\infty(\R\times\R/\Z,\R^+)$ be a smooth, compactly supported, non-negative  function
		 which is positive on $[\tfrac12, \tfrac32]\times \{0\}$. For every $(s,\theta)\in \R\times\R/\Z\equiv \Gamma$ define
		\[
		C(s,\theta):=\int_\R\Psi_0\circ\Phi^t_X(s,\theta)dt\, .
		\]
Note that every orbit intersects $[\tfrac12, \tfrac32]\times \{0\}$, so  $C$ is positive. Also the function $C$  is $\Phi_V^t$-invariant and smooth. Actually, $C(s,\theta)$ is equal to some constant $C$, because $\Phi^t_X|_{\Gamma}$ does not have any non trivial continuous first integral.
		Define then the function $\Psi\in C^\infty(\R\times \R/\Z, \R^+)$ as
		\[
		(s,\theta)\in\R\times\R/\Z\mapsto \Psi(s,\theta):=\dfrac{\Psi_0(s,\theta)}{C}\in \R^+\, .
		\]
		It is obvious that the mean  of $\Psi$  along the orbits is one, 		as desired.
	\end{proof}

Note that for every $(s,\theta)\in \R\times \R/\Z$, the orbit of $(s,\theta)$ intersects $L$ at a set $(\varphi(s,\theta)+\Z)\times \{0\}$.   This defines a $\Phi_X^t$-invariant map:
\[\varphi :  \R\times \R/\Z\to \R/\Z.\]
Note that  for every $s,t\in\R$ it holds
	\begin{equation}\label{def:varphi} 	\varphi\circ\Phi^t_X(s,0)=s \mod 1\, .
	\end{equation}
We recall that $\Sigma$ is a Poincar\'e section containing $L$ and transverse to $\Gamma$ and that $\Omega$ denote an invariant  symplectic form for the first return map $f$.
Let $N:\R\times\R/\Z\equiv\Gamma\to \R^3$ be a  vector field such that: \cn
\begin{equation}\label{def N}  T\Gamma\oplus \R N= \R^3\qand \Omega(\partial_t \ell(s),  N_\Sigma\circ \ell(s))= 1\quad \forall s\in \R \; ,\end{equation}
 where $N_\Sigma$ denotes the projection of $N|_\Gamma$ into the section $\Sigma$.

		Now for every  $g\in C^\infty(\R/\Z,\R)$ we define  the smooth vector field $\cal I (g)\in C^\infty(\R\times\R/\Z,\R^3)$ as:
	\begin{equation}\label{def of O}\cal I (g):= (s , \theta)\in \Gamma  \mapsto  g\circ\varphi  (s,\theta)\cdot \Psi(s,\theta)\cdot N(s, \theta)\in\R^3.\end{equation}
	
	We shall prove that the operator which associates to  $g$ the Melnikov function of $\cal I (g)$ is the identity. In order to do this, let us  go back to the two dimensional setting.
	Let $\epsilon(g)\in C^\infty(\Sigma, \R^2)$ be the perturbation induced by the vector field $X+\cal I(g)$ on the first return map $f$. Consider the smooth family $r\mapsto \epsilon(rg)$. Then the partial derivative $\partial_r\epsilon|_{r=0}$ is:
\[\partial_r\epsilon(rg)|_{r=0}=\lim_{r\to 0} \frac1r \epsilon(r\cdot g) \, .
	\]
	
	\begin{lemma}\label{lemma Omega} For every $s\in\R$
		it holds:
		\[
		\Omega(\partial_t\ell(s+1), \partial_r\epsilon(rg)|_{r=0}\circ \ell (s)\cn)=g(s)\cdot  \left(\int_0^{1}\Psi\circ \Phi^t_X(s,0)dt\right)\, ,
		\]
		where $g$ is identified with a $1$-periodic real function.
	\end{lemma}
	
	\begin{proof}
A flow-box coordinate neighborhood $V$ of $\Gamma$ is diffeomorphic to $\Gamma\times (-1,1)$. This  diffeomorphism induces an identification of $V$ with $\R \times \R/\Z\times (-1,1)$ and can be chosen so that  $N$ is constantly equal to $(0,0,1)$, $\Sigma\cap V$ is identified with $\R\times\{0\}\times \R$ and $X$ has zero third component.

Let $X_g:= X+{\cal{I}}(g)$ and denote by $(\Phi_{X_g}^t)_{t\in\R}$ its  flow. Let $\tau_g:\Sigma\to\R$ be the first return time on $\Sigma$ associated to $X_g$. In the coordinates given by the identification, we have:
\[X_g(s,t)= X(s,\theta)+(0,0,  g\circ\varphi  (s,\theta)\cdot \Psi(s,\theta))\; ,\quad \forall (s,\theta)\in \R\times \R/\Z\; .\]

Let $z=\ell(s)$ and  express the perturbation $\epsilon(g)$ associated to $\cal I (g)$ as
		\[
		\epsilon(g)(z)=\left((f+\epsilon(g))(z)-z\right)-(f(z)-z)\,.\]
		In the coordinates of the tubular neighborhood of $\Gamma$ we have:
		\[(0,\epsilon(g)(z))=\int_0^{\tau_g(z)}(X+\cal I(g))\circ\Phi^t_{X+\cal I(g)}(z)dt-\int_0^{1}X\circ\Phi^t_X(z)dt\, ,
		\]
		where the coordinate $0$ is the coordinate in $\R/\Z$ (we have changed the ordering of the coordinates for the ease of notation).
		
		Hence for every $r\ge 1$:
\[ (0,\epsilon(g)(z))=
\int_1^{\tau_g(z)}X\circ\Phi^t_{X}(z)dt+
\int_0^{1} \cal I(g) \circ\Phi^t_{X}(z)dt+
\int_0^{1}\left(X\circ\Phi^t_{X+\cal I(g)}(z)-X\circ\Phi^t_X(z)\right)dt+
O(\|g\|^2_{C^r})\, ,\]
where we have used that $\int_1^{\tau_g(z)}X\circ\Phi^t_{X+\cal I(g)}(z) dt = \int_1^{\tau_g(z)} X\circ\Phi^t_X(z)dt+ O(\Vert g\Vert^2_{C^r})$.
Since, with $p_3: \R\times \R/\Z\times (-1,1)\to (-1,1)$ the third coordinate projection, it holds $p_3(X)=0$ on the neighborhood $V$, the first and third terms have their third components null while the second term has its first and second components null. Thus, it holds:
\[ p_3(0,\epsilon(g)(z))= p_3\left(
\int_0^{1} \cal I(g) \circ\Phi^t_{X}(z)dt\right)+
O(\|g\|^2_{C^r})\;
.\]
Now we plug \cref{def of O} to obtain:
  \[ p_3(0,\epsilon(g)(z))= p_3\left(
\int_0^{1} g\circ\varphi  \circ\Phi^t_{X}(z)\cdot \Psi \circ\Phi^t_{X}(z)\cdot N  \circ\Phi^t_{X}(z)dt\right)+
O(\|g\|^2_{C^r})\;
.\]
Now we use that $p_3\circ N\equiv  1$ and that, from \eqref{def N}, $\Omega(\partial_t \ell(s+1) , \epsilon(g)\circ \ell(s))=  p_3(0,\epsilon(g)(z))$, to obtain:
\begin{equation}\label{eq_melni} \Omega(\partial_t \ell (s+1), \epsilon(g) \circ \ell(s))=
\int_0^{1} g\circ\varphi  \circ\Phi^t_{X}(z)\cdot \Psi \circ\Phi^t_{X}(z)dt +
O(\|g\|^2_{C^r})\;
.\end{equation}
Using the invariance of $\varphi$ by the flow, we have
$ g\circ\varphi  \circ\Phi^t_{X}(z)=  g\circ\ell(s)$ for every $t$. Using Eq.\eqref{eq_melni} with the perturbation $\epsilon(rg)$, taking its derivative with respect to $r$ and evaluating at $r=0$ yield the sought equality.
%
	\end{proof}
	
	The Melnikov function associated to $\cal I(g)$ is
	\[
	\cal M(\cal I(g))(s)=\cal M(\epsilon(g))(s)=\sum_{k\in\Z}\Omega(\partial_t\ell(s+k),\partial_r\epsilon(rg)|_{r=0}(\ell(s+k-1))).
	\]
	By \cref{lemma Omega}, since the function $g$ is $1$-periodic and by \cref{prop Psi}, we deduce that
	\[
	\cal M(\cal I(g))(s)=g(s)\cdot\sum_{k\in\Z} \left(\int_0^{1}\Psi\circ\Phi^t_X(s+k,0)\,dt\right)
	\]
	and, since $\Phi^t_X(s+k,0)=\Phi^{t+k}_X(s,0)$ for every $k\in\Z$, we conclude that
	\[
	\cal M(\cal I(g))(s)= g(s)\cdot \int_\R \Psi\circ \Phi^t_X(s,0) dt = g(s)\, .
	\]
\end{proof}

\section{Proof of
\cref{cor:CK}: Cauchy-Kovalevskaya's Theorem for curl}\label{section CK and GAT}

\subsection{Application of Cauchy-Kovalevskaya's Theorem for curl}

	Let $\Sigma\subset\R^3$ be a bounded, oriented, analytic surface that can be analytically extended in a neighborhood. Given a vector field $W$ on $\Sigma$, we can decompose it as follows:
	\[
	W=W_T+(W\cdot N)N,
	\]
	where $N$ is a unit normal to $\Sigma$, while $W_T$ is the component of $W$ tangent to $\Sigma$. We refer to $W$ as the Cauchy datum for the Cauchy-Kovalevskaya theorem. Denote as $g:=W\cdot N$ the normal component of the Cauchy datum.
	
	We will be interested in $W$ satisfying the following condition:
	\begin{equation}\tag{$*$}\label{CK}
	d(W^\flat_T)=g\sigma,
	\end{equation}
	where $W^\flat_T$ is the 1-form dual to $W_T$, push-forwarded on $\Sigma$, and $\sigma$ is the area form on $\Sigma$ induced from the ambient space $\R^3$.
	The following   generalizes Theorem 3.1 in \cite{EncPS}.
	\begin{theorem}[Cauchy-Kovalevskaya's Theorem for curl]\label{CK:teo}
		Let $\Sigma\subset\R^3$ be a surface as above and let $W$ be the analytic Cauchy datum in $C^\omega(\bar\Sigma,\R^3)$. The equation
		\begin{align}
		\mathrm{curl }\, X &= X, \\
		X|_{\Sigma}&=W
		\end{align}
		has a unique, analytic solution in a neighborhood $\Omega$ of $\Sigma$ if and only if $W$ fulfills condition \eqref{CK}.
	\end{theorem}
	
This theorem implies \cref{cor:CK}
\begin{proof}[Proof of \cref{cor:CK}] Let $\Gamma$ be an analytic surface in $\R^3$  whose closure $\bar \Gamma$ is diffeomorphic to $\I \times \R/\Z$. We are going to show that for every $g\in C^\omega(\bar \Gamma,\R)$, there exists a neighborhood $\Omega$ of $\bar \Gamma$ in $\R^3$ and a Beltrami field $X\in C^\omega(\Omega,\R^3)$ such that the normal component to $\Gamma$ of $X|_\Gamma$ equals $g$. To be precise, 	with  $N$ a unit normal vector field on $\Gamma$, this means that $(X|_\Gamma)\cdot N= g$.

In order to show this, it suffices to observe that, given the normal component $g\in C^\omega(\bar\Gamma,\R)$ of a Cauchy datum on $\Gamma$, there are many Cauchy data $W$ satisfying condition \eqref{CK} and such that $W\cdot N=g$. Endow $\bar \Gamma\simeq \I \times \R/\Z$ with coordinates $(z,\theta)\in\I\times\R/\Z$ so that the area form $\sigma$ is $dz\wedge d\theta$. We are looking for a Cauchy datum $W$ such that
	\[
	d(W^\flat_T)=gdz\wedge d\theta.
	\]
	Defining
	\[
	\tilde{g}(z,\theta):=\int_{-1}^zg(s,\theta)ds,
	\]
	we set
	\[
	W^\flat_T:=\tilde{g}(z,\theta)d\theta.
	\]
	Thus, it clearly holds that $d(W^\flat_T)=g(z,\theta)dz\wedge d\theta$, that is $W$ satisfies condition \eqref{CK}. By \cref{CK:teo} applied at $\Gamma$, we conclude that there exists a neighborhood $\Omega$ of $\Gamma$ in $\R^3$ and an Euclidean Beltrami field $X(g)\equiv X\in C^\omega(\Omega,\R^3)$, defined as
	\[
	X|_{\Gamma}:= gN+(\tilde g(z,\theta)d\theta)^\#,
	\]
	whose normal component is $g$.
\end{proof}
\begin{remark}\label{rmk choose}
	Any other $W^\flat_T$ satisfying condition \eqref{CK} is of the form $\tilde{g}(z,\theta)d\theta+\alpha$, where $\alpha$ is a closed 1-form on $\Sigma$. In the following, we will always choose $\alpha=0$, which implies the uniqueness of the Beltrami field with normal datum $g$.
\end{remark}

\begin{remark}\label{linearity CK}
	Let $g\in C^\omega(\bar\Sigma,\R)\mapsto X(g)\in C^\omega(\Omega,\R^3)$ be the map that associates to every normal datum $g$ the corresponding Beltrami field $X(g)$ given by \cref{cor:CK}. Such a map is linear.
	Indeed, let $g_1,g_2$ be normal data and $c_1,c_2\in\R$. Then, for $i=1,2$, there exists a neighborhood $\Omega_i$ of $\Sigma$ and a Beltrami field $X^i:=X(g_i)\in C^\omega(\Omega_i,\R^3)$ such that
	\[
	X^i|_{\Sigma}=W^i_T+g_iN
	\]
	where $W^i_T:=(\tilde{g}_i(z,\theta)d\theta)^\#$. Thus, the Beltrami field $X(c_1 g_1+c_2 g_2)\in C^\omega(\Omega_1\cap\Omega_2,\R^3)$ satisfies
	\[
	X(c_1 g_1+c_2 g_2)|_{\Sigma}=:X=\]
	\[(c_1\tilde{g}_1(z,\theta)d\theta)^\#+(c_2\tilde{g}_2(z,\theta)d\theta)^\#+(c_1 g_1+c_2 g_2)N=(c_1 W^1_T+c_2 W^2_T)+(c_1 g_1+c_2 g_2)N=
	\]
	\[
	c_1(W^1_T+g_1N)+c_2(W^2_T+g_2N)=c_1 X^1+c_2 X^2.
	\]
	Since the equation $\mathrm{curl }\, X=X$ is linear, by uniqueness (see \cref{rmk choose}) we have that $X$ associated to the normal datum $c_1g_1+c_2g_2$ is
	\[
	X=c_1 X^1|_{\Omega_1\cap\Omega_2}+c_2 X^2|_{\Omega_1\cap\Omega_2}.
	\]
\end{remark}

\subsection{Proof of Theorem \ref{CK:teo}}\label{proofCK}

Let us first consider a toy model to understand \cref{CK:teo} and Condition \eqref{CK}, following \cite[Subsection 3.2]{Survey-EPD}.
Let us consider $\Sigma=\{x_3=0\}$ with coordinates $(x_1,x_2,x_3)$ in $\R^3$. Fix an analytic Cauchy datum
\[
W=W_1(x_1,x_2)\partial_{x_1}+W_2(x_1,x_2)\partial_{x_2}+W_3(x_1,x_2)\partial_{x_3}.
\]
In coordinates, the Beltrami equation $\mathrm{curl}\ X= X$ reads as
\begin{numcases}{}
\frac{\partial X_1}{\partial x_3}\qquad = & $ X_2+\frac{\partial X_3}{\partial x_1} $\label{first:Bel}
\\
\frac{\partial X_2}{\partial x_3}\qquad =&$- X_1+\frac{\partial X_3}{\partial x_2}$\label{second:Bel}
\\
\frac{\partial X_3}{\partial x_3}\qquad =&$-\frac{\partial X_1}{\partial x_1}-\frac{\partial X_2}{\partial x_2}$ \label{third:Bel}.
\end{numcases}
By the standard Cauchy-Kovaleskaya's Theorem applied at $\Sigma=\{x_3=0\}$, there exists a unique analytic solution to this system in a neighborhood $\Omega$ of $\Sigma$, with Cauchy datum $W$.

Considering $\frac{\partial}{\partial x_1}$\eqref{second:Bel}$-\frac{\partial}{\partial x_2}$\eqref{first:Bel} and using \eqref{third:Bel}, we obtain
\[
\frac{\partial}{\partial x_3}\left(\frac{\partial X_2}{\partial x_1}-\frac{\partial X_1}{\partial x_2}- X_3\right)=0.
\]
Thus
\begin{equation}\label{eq:toy}
\frac{\partial X_2}{\partial x_1}-\frac{\partial X_1}{\partial x_2}= X_3+f(x_1,x_2),
\end{equation}
in $\Omega$, for some analytic function $f$. Evaluating \eqref{eq:toy} at $\Sigma=\{x_3=0\}$, we obtain
\[
\frac{\partial W_2}{\partial x_1}-\frac{\partial W_1}{\partial x_2}=W_3+f(x_1,x_2).
\]
The vector field $X=(X_1,X_2,X_3)$ is Beltrami in $\Omega$ if and only if the Cauchy datum satisfies the constraint
\begin{equation}\label{const:CK}
\frac{\partial W_2}{\partial x_1}-\frac{\partial W_1}{\partial x_2}=W_3.
\end{equation}
In terms of the dual 1-form $\beta:=W_1dx_1+W_2dx_2$, this is equivalent to
\begin{equation}\label{const:form:CK}
d\beta=W_3dx_1\wedge dx_2.
\end{equation}

We have to choose the tangent part $\beta$ of the Cauchy datum so that \eqref{const:form:CK} is fullfilled. If $\Sigma$ is open, given $W_3$, this is always possible.\\

We now prove \cref{CK:teo}, whose proof follows closely that of Theorem 3.1 in \cite{EncPS}. Observe that Condition \eqref{CK} reduces to the condition of \cite[Theorem 3.1]{EncPS} when $g=0$.
\begin{proof}[Proof of \cref{CK:teo}]
	Let $\Sigma\subset\R^3$ be a bounded, oriented, analytic surface. Denote by
	\[
	j_\Sigma:\Sigma\to\R^3
	\]
	the analytic embedding of $\Sigma$ into $\R^3$.
	Let $W$ be an analytic vector field defined on $\Sigma$. We will refer to it as Cauchy datum. We write $W$ as
	\[
	W=W_T+(W\cdot N) N,
	\]
	where $N$ is the unit normal vector to $\Sigma$, while $W_T$ is the component of $W$ tangent to $\Sigma$.
	Let $W^\flat$ be the dual 1-form of $W$, while $W_T^\flat$ denotes the dual 1-form of $W_T$. Let $g:=W\cdot N$.
	Recall that the vector field $W$ satisfies Condition \eqref{CK} if
	\[
	d W_T^\flat =g\sigma,
	\]
	where here $d$ is the exterior derivative on $\Sigma$ and $\sigma$ is the induced area-form on $\Sigma$.

	Take local analytic coordinates in a neighborhood $\Omega$ of $\Sigma$:
	\[
	(\rho,\xi_1,\xi_2),
	\]
	where $\rho$ is the signed distance function from $\Sigma$: if $\Omega$ is narrow enough, then this function is analytic.
	In these coordinates, the Euclidean metric reads as:
	\[
	ds^2=d\rho^2+h_{ij}(\rho,\xi_1,\xi_2)d\xi_id\xi_j,
	\]
	where the last expression has to be understood as sum for $i,j=1,2$.  In the sequel, we will use such Einstein convention, i.e. repeated indices have to be understood as summed.
	Denote as $h^{ij}$ the inverse matrix of $h_{ij}$ and $|h|:=\mathrm{det}(h_{ij})$ its determinant.
	
	The Beltrami field $X$ reads in these coordinates as
	\[
	X=a(\rho,\xi_1,\xi_2)\partial_{\rho}+b^i(\rho,\xi_1,\xi_2)\partial_{\xi_i}
	\]
	for some functions $a(\rho,\xi_1,\xi_2), b^1(\rho,\xi_1,\xi_2),b^2(\rho,\xi_1,\xi_2)$.
	The 1-form associated to $X$ is so
	\[
	\beta:=a(\rho,\xi_1,\xi_2)d\rho+c_i(\rho,\xi_1,\xi_2)d\xi_i,
	\]
	with $c_i:=h_{ij}b^j$. In terms of the 1-form $\beta$, the equation $\mathrm{curl}\,X=X$ reads as
	\begin{equation}\label{1form:curl}
	\star d\beta= \beta,
	\end{equation}
	where $\star$ denotes the Hodge star operator. In the coordinates $(\rho,\xi_1,\xi_2)$, the term $\star d\beta$ has the form
	\[
	\star d\beta=\]
	\[\dfrac{1}{|h|^{1/2}}\left( \dfrac{\partial c_2}{\partial\xi_1}-\dfrac{\partial c_1}{\partial \xi_2} \right)d\rho+|h|^{1/2}h^{2i}\left( \dfrac{\partial a}{\partial \xi_i} -\dfrac{\partial c_i}{\partial \rho}\right)d\xi_1+|h|^{1/2}h^{1i}\left( \dfrac{\partial c_i}{\partial\rho}-\dfrac{\partial a}{\partial\xi_i} \right)d\xi_2.
	\]
	The equation $\star d\beta=\beta$ is then equivalent, in $(\rho,\xi_1,\xi_2)$-coordinates, to:
	\begin{numcases}{}
	\dfrac{1}{|h|^{1/2}}\left( \dfrac{\partial c_2}{\partial\xi_1}-\dfrac{\partial c_1}{\partial\xi_2} \right)\qquad = & $ a$\label{first:CKgen}
	\\
	|h|^{1/2}h^{2i}\left( \dfrac{\partial a}{\partial\xi_i}-\dfrac{\partial c_i}{\partial\rho} \right)\qquad =&$c_1$\label{second:CKgen}
	\\
	|h|^{1/2}h^{1i}\left( \dfrac{\partial c_i}{\partial\rho}-\dfrac{\partial a}{\partial\xi_i} \right)\qquad =&$c_2$ \label{third:CKgen}
	\end{numcases}
	
	We start by showing the \emph{necessity} of Condition \eqref{CK}. The tangential component of $X$ on $\Sigma$ is
	\[
	X_T=b^i(0,\xi_1,\xi_2)\partial_{\xi_i}.
	\]
	Its dual 1-form is then $X_T^\flat=c_i(0,\xi_1,\xi_2)d\xi_i$. 
	Thus
	\[
	d(c_i(0,\xi_1,\xi_2)d\xi_i)=\left( \dfrac{\partial c_2(0,\xi_1,\xi_2)}{\partial\xi_1}-\dfrac{\partial c_1(0,\xi_1,\xi_2)}{\partial\xi_2} \right)d\xi_1\wedge d\xi_2.
	\]
	On the other hand the area-form $\sigma$ on $\Sigma$ is, since $N=\partial_\rho$,
	\[
	\sigma=|h|^{1/2}(0,\xi_1,\xi_2)d\xi_1\wedge d\xi_2.
	\]
	Evaluating then \eqref{first:CKgen} at $\rho=0$, i.e. at $\Sigma$, we obtain
	\[
	\dfrac{\partial c_2(0,\xi_1,\xi_2)}{\partial\xi_1}-\dfrac{\partial c_1(0,\xi_1,\xi_2)}{\partial\xi_2}=a(0,\xi_1,\xi_2)|h|^{1/2}(0,\xi_1,\xi_2),
	\]
	and, denoting $g:=a(0,\xi_1,\xi_2)$, we obtain Condition \eqref{CK}:
	\[
	d(X_T^\flat)=g\sigma.
	\]
	
	We now show the \emph{sufficiency} of Condition \eqref{CK}. Consider the auxiliary problem:
	\begin{equation}\tag{A}\label{aux}
	\begin{cases}
	(d+\delta)\psi=\star\psi\quad\text{in }\Omega\\
	\psi|_\Sigma=W^\flat,
	\end{cases}
	\end{equation}
	where $W^\flat$ is the dual 1-form of $W$, $\delta$ is the codifferential (i.e. for $\omega\in\Omega^k$ it holds $\delta\omega=(-1)^{k}\star d\star\omega$), and, denoting as $\Omega^k$ the space of k-forms, $\psi\in\Omega^0\oplus\Omega^1\oplus\Omega^2\oplus\Omega^3$. In particular, $\psi$ can be decomposed as $\psi^0\oplus\psi^1\oplus\psi^2\oplus\psi^3$, with $\psi^k$ a $k$-form. The operator $d+\delta$ is a first order elliptic operator (notice that $(d+\delta)^2=d\delta+\delta d=\Delta$ is the Hodge Laplacian). So, by the standard Cauchy-Kovalevskaya's Theorem (see \cite[Page 502, Proposition 4.2]{PDEbasictheory}), the Cauchy problem \eqref{aux} has a unique analytic solution in $\Omega$, for some $\Omega$ narrow enough.
	
	\begin{lemma}\label{Lemma CK1}
		If $W$ satisfies Condition \eqref{CK}, then $\delta\psi=0$ in $\Omega$.
	\end{lemma}
	\begin{proof}[Proof of Lemma \ref{Lemma CK1}]
		In \cite[Section 4]{EncPS} the authors define a parity operator
		\[
		Q:\Omega^0\oplus\Omega^1\oplus\Omega^2\oplus\Omega^3\to \Omega^0\oplus\Omega^1\oplus\Omega^2\oplus\Omega^3
		\]
		as
		\[
		Q\psi:=\psi^0\oplus -\psi^1\oplus\psi^2\oplus-\psi^3,
		\]
		where $\psi=\psi^0\oplus\psi^1\oplus\psi^2\oplus\psi^3$.
		First, since $\star d\star=-Q\delta$, applying $\star d$ at the equation \eqref{aux} and since $\star\star=\mathrm{Id}$, we  notice that $\delta\psi$ satisfies the elliptic equation
		\[
		(d+\delta)\delta\psi=-\star Q\delta\psi.
		\]
		Therefore, if we show that $\delta\psi|_\Sigma=0$, then $\delta\psi=0$ in $\Omega$ by uniqueness of the solution by the Cauchy-Kovalevskaya's Theorem, since the operator $d+\delta$ is elliptic.
		
		Our aim is then proving that $\delta\psi|_\Sigma=0$. Write $\psi$ as follows:
		\[
		\psi=\Psi+d\rho\wedge\tilde\Psi,
		\]
		where $\Psi,\tilde\Psi$ are forms in $\Omega^0\oplus\Omega^1\oplus\Omega^2\oplus\Omega^3$ such that $i_{\partial_\rho}\Psi=i_{\partial_\rho}\tilde\Psi=0$. Since $\psi|_\Sigma=W^\flat$, then
		\[
		\Psi|_\Sigma=c_id\xi_i\quad\text{and}\quad\tilde\Psi|_\Sigma=g.
		\]
		From \eqref{aux}, it holds that
		\begin{equation}\label{comput1}
		\delta\psi|_\Sigma=(\star\psi-d\psi)|_\Sigma.
		\end{equation}
		By the Cauchy condition, it holds that $\star\psi|_\Sigma=\star W^\flat|_\Sigma$.
		Using that $\star d\rho=|h|^{1/2}d\xi_1\wedge d\xi_2$, we deduce that
		\[
		\star\psi|_\Sigma=\star W^\flat|_\Sigma=\star(gd\rho+c_1(0,\xi_1,\xi_2)d\xi_1+c_2(0,\xi_1,\xi_2)d\xi_2)=\]
		\begin{equation}\label{comput2}
		g|h|^{1/2}d\xi_1\wedge d\xi_2 +d\rho\wedge\tilde\alpha,
		\end{equation}
		for some 1-form $\tilde\alpha$. Denoting as $\bar d$ the exterior derivative only in the variables $\xi_1,\xi_2$, we have
		\[
		d\psi=\bar d\Psi+d\rho\wedge\dfrac{\partial\Psi}{\partial\rho}-d\rho\wedge\bar d\tilde\Psi=\bar d\Psi+d\rho\wedge\left( \dfrac{\partial\Psi}{\partial\rho}-\bar d\tilde\Psi \right).
		\]
		Since $\Psi|_\Sigma=c_1(0,\xi_1,\xi_2)d\xi_1+c_2(0,\xi_1,\xi_2)d\xi_2$, it holds, using Condition \eqref{CK},
		\[
		\bar d\Psi|_\Sigma=g|h|^{1/2}d\xi_1\wedge d\xi_2.
		\]
		Consequently
		\begin{equation}\label{comput3}
		d\psi|_\Sigma=g|h|^{1/2}d\xi_1\wedge d\xi_2+d\rho\wedge\hat\alpha,
		\end{equation}
		for some $\hat\alpha=\hat\alpha^0\oplus\hat\alpha^1\oplus\hat\alpha^2$ on $\Sigma$.
		Putting together \eqref{comput1}, \eqref{comput2} and \eqref{comput3}, we deduce that
		\[
		\delta\psi|_\Sigma=d\rho\wedge\check\alpha,
		\]
		for some $\check\alpha=\tilde\alpha-\hat\alpha=\check\alpha^0\oplus\check\alpha^1\oplus\check\alpha^2$ on $\Sigma$. Noticing that $\psi|_\Sigma=W^\flat$ is a 1-form, we obtain:
		\begin{itemize}
			\item[$(i)$] since $\delta\psi|_\Sigma=d\rho\wedge\check\alpha$ does not contain 0-forms, the components $\delta\psi^0|_\Sigma$ and $\delta\psi^1|_\Sigma$ are both zero;
			\item[$(ii)$] for $p=2,3$ we have \[\delta\psi^p|_\Sigma=(-1)^p(\star d\star\psi^p)|_\Sigma=(-1)^p\star(d\rho\wedge\beta)\] for some form $\beta$. This comes from the fact that
			\[
			d(\star\psi^2)|_\Sigma=d\rho\wedge\partial_\rho(\star\psi^2)|_{\rho=0}+\bar d(\star\psi^2)|_{\rho=0},
			\]
			but since $\psi^2|_{\rho=0}=0$, then $\bar d(\star\psi^2)|_{\rho=0}=\bar d(\star\psi^2|_{\rho=0})=0$. Similarly, since $\psi^3|_\Sigma=0$, it holds that
			\[
			d(\star\psi^3)|_\Sigma=d\rho\wedge\partial_\rho(\star\psi^3)|_{\rho=0}.
			\]
			Therefore, for $p=2,3$, $\delta\psi^p|_\Sigma$ only contains elements with $d\xi_1$ and $d\xi_2$, for $p=2$, and only $d\xi_1\wedge d\xi_2$ for $p=3$.
		\end{itemize}
		The fact that $\delta\psi|_\Sigma$ does not contain the element $d\rho$ contradicts the expression $\delta\psi|_\Sigma=d\rho\wedge\check\alpha$ previously found, unless
		\[
		\delta\psi|_\Sigma=0,
		\]
		as we wanted to prove. This complete the proof of Lemma \ref{Lemma CK1}.
	\end{proof}
	Going back to the problem \eqref{aux}, since $\delta\psi=0$ in $\Omega$, we have that
	\begin{equation*}
	\begin{cases}
	d\psi=\star\psi\quad\text{in }\Omega\\
	\psi|_\Sigma=W^\flat
	\end{cases}
	\end{equation*}
	and, using the identity $\star(\star\psi)=\psi$, it holds
	\begin{equation*}
	\begin{cases}
	\star d\psi=\psi\quad\text{in }\Omega\\
	\psi|_\Sigma=W^\flat.
	\end{cases}
	\end{equation*}
	Since $\star d$ maps 1-forms into 1-forms and $W^\flat$ is a 1-form, it holds $\psi^0=\psi^2=\psi^3=0$. Hence
	\begin{equation*}
	\begin{cases}
	\star d\psi^1=\psi^1\quad\text{in }\Omega\\
	\psi^1|_\Sigma=W^\flat.
	\end{cases}
	\end{equation*}
	Taking the vector field $X$ dual to $\psi^1$, it satisfies then $\mathrm{curl}\,X=X$ and the sufficiency follows. This concludes the proof of the sufficiency and so of \cref{CK:teo}.
\end{proof}

\section{Disjoint union of compact subsets with connected complement}
\label{appendix classical topo}

The following result is used in the proof of Corollary \ref{GATcoro}. It is probably standard, but we include a proof for the sake of completeness.

\begin{proposition}
Let $K_+$ and $K_-$ be two disjoint compact subsets of $\R^n
$ with connected complement. Then $K=K_+\sqcup K_-$ has connected complement.
\end{proposition}

\begin{proof}
Since $K$ is compact, it is clear that $\mathbb R^n\backslash K$ is connected if and only if $\mathbb S^n\backslash K$ is connected, where the $n$-sphere $\mathbb S^n$ is understood as the one-point compactification of $\mathbb R^n$. Alexander's duality then implies \cite{Massey} that
\[
\tilde H_0(\mathbb S^n\backslash K;\mathbb Z)=\tilde H^{n-1}_{CE}(K;\mathbb Z)\,,
\]
where $\tilde H_0(\mathbb S^n\backslash K;\mathbb Z)$ stands for the reduced singular homology of $\mathbb S^n\backslash K$ with integer coefficients (well defined because it is a manifold) and $\tilde H^{n-1}_{CE}(K;\mathbb Z)$ denotes the reduced \v{C}ech cohomology, which is defined for any compact subset of $\mathbb R^n$. Now, since the reduced $k$-th cohomology group coincides with the cohomology group for $k\geq 1$ (this also holds for \v{C}ech cohomology, for which one can also write a long exact sequence \cite{EilSte52}) and $K_+$ and $K_-$ are disjoint, we infer that
\[
\tilde H^{n-1}_{CE}(K;\mathbb Z)=H^{n-1}_{CE}(K;\mathbb Z)=H^{n-1}_{CE}(K_+;\mathbb Z)\oplus H^{n-1}_{CE}(K_-;\mathbb Z)\,,
\]
where to write the last isomorphism we have used the standard Mayer-Vietoris sequence for \v{C}ech cohomology \cite{MIT}. Then, using again Alexander's duality and that $H^{n-1}_{CE}(K_{\pm};\mathbb Z)=\tilde H^{n-1}_{CE}(K_{\pm};\mathbb Z)$, we obtain that
\[
H^{n-1}_{CE}(K_+;\mathbb Z)\oplus H^{n-1}_{CE}(K_-;\mathbb Z)=\tilde H_0(\mathbb S^n\backslash K_+;\mathbb Z)\oplus \tilde H_0(\mathbb S^n\backslash K_-;\mathbb Z)=0
\]
because $\mathbb S^n\backslash K_{\pm}$ are connected by assumption, and the $0$-th reduced singular homology group of a manifold is trivial if and only if the manifold is connected. Putting all these computations together we finally conclude that
\[
\tilde H_0(\mathbb S^n\backslash K;\mathbb Z)=0\,,
\]
so $\mathbb S^n\backslash K$ is connected, as we wanted to prove.
\end{proof}

\end{appendix}
~\newline

\noindent\textit{Acknowledgements.} The authors are very grateful to Francisco Romero Ruiz del Portal for pointing us a reference for a proof that a countable union of totally disconnected compact sets is totally disconnected as well, cf. Fact 6.4. 

\bibliographystyle{alpha}
\bibliography{biblio.bib}
\Addresses

\end{document}